\documentclass[12pt,reqno]{amsart}
\usepackage{hyperref}
\hypersetup{colorlinks=true, linkcolor=red, citecolor=blue, urlcolor=blue}

\usepackage{amssymb,amsmath,amsfonts,latexsym,mathrsfs,eucal,amscd, tikz-cd}
\usepackage{bm}
\usepackage{mathtools}
\usepackage{enumerate}
\usepackage{array,graphics,color}
\allowdisplaybreaks
\usepackage{mathrsfs}
\usepackage{mathtools}
\usepackage{blkarray, bigstrut}
\usepackage{csquotes}
\usepackage{ulem}
\usepackage{xcolor}

\numberwithin{equation}{section}

\newtheorem{dfn}{Definition}[section]
\newtheorem{thm}[dfn]{Theorem}
\newtheorem{lma}[dfn]{Lemma}
\newtheorem{ppsn}[dfn]{Proposition}
\newtheorem{crl}[dfn]{Corollary}
\newtheorem{xmpl}[dfn]{Example}
\newtheorem{rmrk}[dfn]{Remark}

\newtheorem{hyp}[dfn]{Hypothesis}

\newcommand{\D}{\mathbb{D}}
\newcommand{\cir}{\mathbb{T}}
\newcommand{\R}{\mathbb{R}}
\newcommand{\C}{\mathbb{C}}
\newcommand{\N}{\mathbb{N}}	
\newcommand{\cL}{\mathcal{L}}	
\newcommand{\Z}{\mathbb{Z}}
\newcommand{\hilh}{\mathcal{H}}
\newcommand{\hilk}{\mathcal{K}}
\newcommand{\bh}{\mathcal{B}(\hilh)}

\newcommand{\Sp}{\mathcal{S}}
\newcommand{\Spsa}{\Sp_{\text{sa}}}
\newcommand{\A}{\mathcal{A}}

\newcommand{\Tr}{\operatorname{Tr}}

\newcommand{\conth}{\text{\normalfont Cont}(\hilh)}
\newcommand{\unih}{\mathcal{U}(\hilh)}
\newcommand{\mult}{\text{\normalfont Mult}}
\newcommand{\lin}{\text{\normalfont Lin}}
\newcommand{\Conv}{\mathop{\scalebox{1.5}{\raisebox{-0.2ex}{$\ast$}}}}

\begin{document}
\title[Differentiability of functions of contractions and trace formulas]{Trace formulas for $\mathcal{S}^p$-perturbations and extension of Koplienko-Neidhardt trace formulas}

\author[Chattopadhyay]{Arup Chattopadhyay}

\address{Department of Mathematics, Indian Institute of Technology Guwahati, Guwahati, 781039, India}
\email {arupchatt@iitg.ac.in, 2003arupchattopadhyay@gmail.com}

\author[Coine]{Cl\'ement Coine}
\address{Normandie Univ, UNICAEN, CNRS, LMNO, 14000 Caen, France}	
\email{clement.coine@unicaen.fr}

\author[Giri]{Saikat Giri}
\address{Department of Mathematics, Indian Institute of Technology Guwahati, Guwahati, 781039, India}
\email{saikat.giri@iitg.ac.in, saikatgiri90@gmail.com}

\author[Pradhan]{Chandan Pradhan}
\address{Department of Mathematics and Statistics, University of New Mexico, Albuquerque, NM 87106, United States}
\email{chandan.pradhan2108@gmail.com, cpradhan@unm.edu}



\begin{abstract}
In this paper, we extend the class of admissible functions for the trace formula of the second order in the self-adjoint, unitary, and contraction cases for a perturbation in the Hilbert-Schmidt class $\Sp^2(\hilh)$ by assuming a certain factorization of the divided difference $f^{[2]}$. This class is the natural one to ensure that the second order Taylor remainder is a trace class operator. These results substantially extend analogous results of \cite{CoLemSkSu19}, \cite{ChPrSk24}, \cite{ChDaPr24} and \cite{PoSu12}. Secondly, for a Schatten $\Sp^p$-perturbation, $1<p<\infty$, we prove general modified trace formulas for every $n$-times differentiable functions with bounded $n$-th derivative in the self-adjoint and unitary cases and for every $f$ such that $f$ and its derivatives are in the disk algebra $\A(\D)$ in the contraction case. En route, we establish the higher order differentiability of functions of contractions for such $f$, thereby generalizing an old result in \cite{Pe09}.
\end{abstract}

\maketitle
{
\hypersetup{linkcolor=black}
\tableofcontents
}

\section{Introduction}

Let $\hilh$ be a separable complex Hilbert space and let $\bh$ be the space of all bounded linear operators on $\hilh$ equipped with the standard trace $\Tr$. Let, for every $p\in [1,+\infty)$, $\Sp^p(\hilh)$ be the $p$-Schatten class over $\hilh$. Let $H_0$ be a self-adjoint (possibly unbounded) operator in $\hilh$ and let $V=V^* \in \mathcal{S}^1(\hilh)$. If $f$ is from the Wiener class $W_{1}$, that is, $f$ is a function on $\R$ with Fourier transform of $f'$ is integrable, then
$$f(H_0+V)-f(H_0)\in\Sp^1(\hilh),$$
and it was proved by Krein \cite{Kr53} that there exists a function $\xi\in L^1(\R)$, depending on $H_{0}$ and $V$, such that
\begin{equation}\label{Kreinselfadj}
\forall f\in W_1, \quad \Tr(f(H_0+V)-f(H_0))=\int_{\R}f'(t) \xi(t) dt. 
\end{equation}
A similar formula was proved by Krein in \cite{Kr62} for functions defined on the unit circle $\mathbb{T}$ of $\C$. Namely, if $U_0$ is a unitary operator on $\hilh$ and $A=A^* \in \mathcal{S}^1(\hilh)$, there exists a function $\xi \in L^1(\mathbb{T})$ such that, for every function $f : \cir\to\C$ such that $f'$ has an absolutely convergent Fourier series,
\begin{equation}\label{Kreinselfunit}
\Tr(f(e^{iA}U_0)-f(U_0)) = \int_{\mathbb{T}} f'(z) \xi(z) dz. 
\end{equation}
In the formulas \eqref{Kreinselfadj} and \eqref{Kreinselfunit}, the function $\xi$ is called Lifshitz-Krein spectral shift function and was firstly introduced in a special case by Lifshitz \cite{Li52}. It plays an important role in mathematical physics, scattering theory, and noncommutative geometry, see \cite{BiYa93,ChCo97,Connes-Book,Yafaev} and the references therein. For instance, the function $\xi$ appears in the determinant of a scattering matrix in \cite{BiKr62}. The initial trace formulas were derived for self-adjoint and unitary operators, but remarkable extensions have been obtained, e.g., for pairs of contractions with trace class difference, the corresponding first order trace formula was fully resolved in \cite{MaNePe19} by removing earlier restrictions from \cite{La65,MaNe15,Ne88-II}.
For a detailed historical discussion on first order trace formulas, see \cite{ChDaPr24,MaNe15,MaNePe19}.

\medskip

The best result to date concerning the description of the class of functions for which \eqref{Kreinselfadj} holds true for every self-adjoint operators $H_0,V$ in $\hilh$ with $V\in\Sp^1(\hilh)$ is due to \cite{Pe16}, where it is shown to hold true for operator Lipschitz functions, that is, the class of functions $f:\R\to\C$ for which $f(H_0+V)-f(H_0)\in\Sp^1(\hilh)$ for every self-adjoint operators $H_0,V$ in $\hilh$ with $V\in\Sp^1(\hilh)$, see \cite[Theorem 6.1]{Pe16}.
Moreover, we have the following characterization: according to \cite[Theorems 2.2.2, 2.2.3, 3.1.10, 3.3.6]{AlPe16Survey}, a differentiable function $f:\R\to\C$ is operator Lipschitz if and only if there exist a separable Hilbert space $\mathfrak{H}$ and bounded weakly continuous functions $a,b:\R\to\mathfrak{H}$ such that the divided difference $f^{[1]}$ admits the factorization
\begin{equation}\label{introfacto1}
f^{[1]}(t_1,t_2)=\left\langle a(t_1),b(t_2) \right\rangle, \quad \forall~(t_1,t_2)\in \mathbb{R}^2.
\end{equation}
Similar characterizations for the unitary case \eqref{Kreinselfunit} were proved in \cite{AlPe16}, and for functions of contractions in \cite{MaNePe16}.

The various trace formulas mentioned above in the first order have been extended to higher order by considering the Taylor remainders, which are defined in the self-adjoint case by
\begin{align}\label{eq:TR_self}
\mathcal{R}^{\lin}_{n}(f,H_0,V):=f(H_0+V)-f(H_0)-\sum_{k=1}^{n-1}\frac{1}{k!}\frac{d^k}{ds^k}\bigg|_{s=0}f(H_0+sV).
\end{align}
In this formula, the derivatives are taken with respect either to the strong operator topology, the operator norm, or a Schatten norm, depending on which subspace $V$ belongs to and on the regularity of $f$. Note that the existing formulas require strong assumptions on $f$ to ensure that $\mathcal{R}^{\lin}_{n}(f,H_0,V) \in \mathcal{S}^1(\hilh)$.
In the case $n=2$, the idea of considering such a remainder goes back to Koplienko, who proved in \cite{Ko84} that if $V \in \Sp^2(\hilh)$, there exists a unique $\eta\in L^1(\R)$ such that
\begin{equation}\label{traceKop}
\Tr\Big( f(H_0+V)-f(H_0)-\frac{d}{ds}\bigg|_{s=0}f(H_0+sV)\Big)=\int_{\R} f''(x)\eta(x) dx,
\end{equation}
for every rational function $f$ with poles off $\R$. This formula was extended to every $C^2$ function $f$ with bounded derivatives, such that there exist a separable Hilbert space $\mathfrak{H}$ and two bounded Borel functions $a,b:\R^2\to\mathfrak{H}$ such that the second order divided difference $f^{[2]}$ satisfies
\begin{equation}\label{introfacto2}
f^{[2]}(t_1,t_2,t_3)=\left\langle a(t_1,t_2),b(t_2,t_3) \right\rangle, \quad \forall~(t_1,t_2,t_3)\in \mathbb{R}^3
\end{equation}
(see \cite{CoLemSkSu19}). Such factorization is to be compared with \eqref{introfacto1} and originates from the representation of Taylor remainders as triple operator integrals associated to $f^{[2]}$ and the characterization of the functions $f$ such that these operators take values in $\mathcal{S}^1(\hilh)$. See Section \ref{sectionmainassumption} for more details. This class of functions includes the Besov space $B^2_{\infty1}(\R)$, for which \eqref{traceKop} was previously established in \cite[Theorem 4.6]{Pe05}. For a trace formula of \eqref{eq:TR_self} for a general integer $n \geq 2$ and a perturbation $V\in \mathcal{S}^n(\hilh)$, we refer to \cite{AlPe11} and \cite[Theorem 1.1]{PoSkSu13}.

In the cases of unitaries and contractions, the higher order trace formulas take different forms depending on whether the differentiation is performed along a linear or a multiplicative path. In the case of a linear path $s\mapsto T_0+s(T_1-T_0)$ for two contractions $T_0$ and $T_1$ such that $T_1-T_0\in\Sp^n(\hilh)$, the existence of a spectral shift function $\eta_{n}\in L^1(\cir)$ such that, for every polynomial $f$,
$$\Tr\Big(f(T_1)-f(T_0)-\sum_{k=1}^{n-1}\frac{1}{k!}\frac{d^k}{ds^k}\bigg|_{s=0}f(T_0+s(T_1-T_0))\Big)=\int_{\cir} f^{(n)}(z)\eta_{n}(z)dz,$$
was established in \cite{PoSu12} for $n=2$, and in \cite{PoSkSu14} for $n\ge3$. For the multiplicative path $s\mapsto e^{isA}T_{0}$ with $A=A^{*}\in\Sp^{n}(\hilh)$, the best result to date is \cite[Theorem 3.5]{ChPrSk24}. There, the authors proved that there exist $\eta_1,\ldots,\eta_n\in L^1(\cir)$ such that, for every
$n$-times differentiable function $f:\cir\to\C$ satisfying 
$$f^{[n]}\in\mathcal{B}(\cir) \widehat{\otimes}_{i}\cdots\widehat{\otimes}_{i} \mathcal{B}(\cir),$$
where $\widehat{\otimes}_{i}$ denotes the integral projective tensor product and $\mathcal{B}(\cir)$ is the space of bounded measurable function on $\cir$, the following trace formula holds:
$$\Tr\Big(f(e^{iA}T_{0})-f(T_0)-\sum_{k=1}^{n-1}\frac{1}{k!}\frac{d^k}{ds^k}\bigg|_{s=0}f(e^{isA}T_{0})\Big)=\sum_{k=1}^n \int_{\cir} f^{(k)}(z)\eta_k(z)dz.$$

\medskip

As mentioned before, strong assumptions have to be made on the function $f$ to ensure that the associated Taylor remainder is a trace class operator. For instance, even when the function $f \in C^n(\cir)$ (or $C^n(\R)$ with bounded derivatives) with a perturbation in $\Sp^n(\hilh)$, the remainder may not belong to $\Sp^{1}(\hilh)$, see \cite{CLPST1, CLPST2, PoSkSuTo17}. However, with simple and relaxed assumptions, that is, when $f$ is $n$-times differentiable with bounded derivatives, these remainders belong to $\Sp^p(\hilh)$, $1<p<\infty$, provided the perturbation is an element of $\Sp^{np}(\hilh)$, see \cite[Proposition 3.3]{Co22} and \cite[Remark 5.3]{Co24}. With this in mind, modified trace formulas have been investigated by noting that it becomes natural to consider, in the self-adjoint case,
$$\mathcal{R}^{\lin}_{n}(f,H_0,V)\cdot X,$$
where $X\in\Sp^q(\hilh)$, $q$ being the conjugate exponent of $p$. In the unitary and contraction settings with $n=1$, this type of modification appears in \cite{BhChGiPr23}. For instance, for a pair of unitaries $\{U,V\}$, it was proved in \cite[Theorem 3.1]{BhChGiPr23} that for every $f\in C^1(\cir)$ with $\sum_{n\in\Z}^{}|n\widehat{f}(n)|<\infty$, the following identity holds:
$$\Tr\left\{(f(U)-f(V))\cdot X\right\} = \int_0^{2\pi} f'(e^{it})\xi(t)dt$$
where $\xi\in L^1([0,2\pi])$ depends on the triple $\{U,V,X\}$ such that $U-V\in\Sp^2(\hilh)$ and $X\in \Sp^2(\hilh)$. See \cite[Theorem 3.2]{BhChGiPr23} for such result for pair of contractions. For an $n$-th order modified trace formula with a particular choice of $X$ and a smaller set of admissible functions, see \cite[Theorem 4.5]{Sk17Adv}.

The present paper has three main aims:
\begin{enumerate}
\item to extend the class of admissible functions for the usual trace formula in the case $n=2$, for pair of self-adjoint operators, unitaries, and contractions, by considering the class of functions $f$ such that $f^{[2]}$ has a certain factorization as in \eqref{introfacto2}. 
\item to establish modified trace formulas in both the self-adjoint and unitary settings, for $n$-times differentiable functions whose $n$-th derivative is bounded, and for perturbations in $\Sp^p(\hilh)$, $1<p<\infty$. The assumptions on $f$ are necessary to ensure that the Taylor remainders are well defined and elements of $\Sp^p(\hilh)$ for every pair of operators. Hence, these results are optimal in the range $1<p<\infty$. This shows the major difference with the typical trace formula, which does not apply to this functions class.
\item to establish a general modified trace formula for a pair of contractions, at every order $n$, for functions $f$ in the disk algebra $\A(\D)$ such that $f',\ldots,f^{(n)}\in\A(\D)$. To achieve this, we study the differentiability of such functions along a linear path of contractions with $\Sp^p$-perturbation, thereby generalizing the results of \cite{Pe09}.
\end{enumerate}

We now discuss the content of the paper in more detail. In Section \ref{Sec2}, we define multiple operator integrals (MOIs) associated with normal operators, in line with \cite{CoLemSu21}.

In Section \ref{DoFoC}, we extend the definition of MOI to cover the case of semi-spectral measures, which allows us to define MOI for contractions. We then derive several perturbation formulas and estimates, which will be instrumental in establishing the differentiability of functions of contractions, for functions in the disk algebra $\A(\D)$, along a linear path of contractions with $\mathcal{S}^p$-perturbation (see Theorem \ref{Differentiation}).

In Section \ref{Second order trace formulas}, we first provide several sufficient conditions on functions $f:\R~(\mbox{or $\mathbb{T}$})\to\C$, under which the second order divided difference $f^{[2]}$ admits factorization of the form \eqref{introfacto2} (see Propositions \ref{Suff-Ppsn} and \ref{Example-Ppsn}). Then, we establish second order trace formulas for such functions, corresponding to self-adjoint, unitary, and contraction operators. The function class we consider is natural and potentially the best function class. In the case of self-adjoint operators, our function class extends the one considered in \cite{CoLemSkSu19}, and now provides a complete analog to the class considered in \cite{Pe16} for the Krein trace formula (see Theorem \ref{SOT-Sacase}). In the unitary case, our function class includes the Besov space $B^{2}_{\infty1}(\cir)$ studied in \cite{Pe05} and also extends the class considered in \cite{ChPrSk24} (see Theorem \ref{N-Traceformula}). Moreover, our proof is simpler, relying on function approximation techniques, which are made possible thanks to the compactness of the spectrum. In the case of contractions, we prove trace formulas by considering both linear and multiplicative paths.
For linear path, we establish a trace formula for functions $f\in\mathcal{A}(\D)$ such that $f',f''\in\mathcal{A}(\D)$ and $f^{[2]}$ belongs to the projective tensor product $\mathcal{A}(\D)\widehat{{\otimes}}\mathcal{A}(\D)\widehat{{\otimes}}\mathcal{A}(\D)$. Our function class is fairly large and includes all polynomials considered in \cite{PoSu12} (see Theorem \ref{SOT-Concase3}). Additionally, we show that this function class can be further enlarged by imposing certain mild assumptions on the pair of contractions $\{T_{1},T_{0}\}$, such as requiring one of the operators to be normal or strict (see Theorems \ref{SOT-Concase1} and \ref{SOT-Concase2}). We note that similar restrictions on the pair $\{T_{1},T_{0}\}$ were previously considered in \cite{ChDaPr24}. Finally, in Theorem \ref{SOT-Concase-M}, we establish a trace formula for a pair of contractions along a multiplicative path.

In Section \ref{Modified trace formulas}, we prove modified trace formulas for $\Sp^p$-perturbations in the range $1<p<\infty$, in the self-adjoint, unitary, and contraction cases. We conclude this study by analyzing certain properties of the spectral shift functions obtained in modified trace formulas, specifically their non-negativity (see Theorem \ref{Thm-Pos}), continuity and differentiability. Indeed, we establish a result concerning the continuity and differentiability of the spectral shift functions under certain topologies, in the same spirit as \cite{CaGeLeNiPoSu16} (see Proposition \ref{cor:ssf_cont_diff}).

The proofs of aforementioned results will make use of the theory of multiple operator integrals. They proved to be a convenient tool to deal with questions in perturbation theory in the Schatten classes setting. For instance, the derivatives, as well as the Taylor remainders for operator functions are expressed by means of linear combinations of MOI, see \cite{ChCoGiPr24, Co22,Co24} for the self-adjoint and unitary cases, and see \cite{Pe09} for the contraction case which makes use of the unitary case together with an argument of dilation. In the case $n=2$, the class of functions that we consider follows from a characterization of MOI taking values in $\Sp^1(\hilh)$, see \cite{CoLemSu21}, which corresponds to the factorization given in \eqref{introfacto2}. Additionally, in the case of contractions, we use a finite-dimensional approximation technique introduced by Voiculescu \cite{Vo79} to prove trace formulas. For $1\le n<p<\infty$, the modified trace formula will be obtained in full generality, thanks to the optimal $\Sp^{p/n}$-estimates for the Taylor remainders. The various norm estimates (for both the usual and modified trace formulas) will in turn yield estimates for the trace, allowing its representation as an integral with respect to a finite measure. The remaining argument for each trace formula in this paper will consist in showing that such measure is absolutely continuous with respect to the Lebesgue measure, whether by finite-dimensional approximation (see the proof of Theorem \ref{SOT-Concase1}) or by induction, using the result for $n-1$ to construct the desired integrable function.

\medskip

\noindent\textbf{Notations and conventions.} We will give here a few notations that will be used throughout this article. Any notations mentioned later but not covered here are standard.
\begin{enumerate}[-]
\item Throughout this article, $\hilh$ denotes a complex separable Hilbert space. $\Sp^{p}(\hilh)$ will denote the $p$-Schatten class on $\hilh$ endowed with the $\|\cdot\|_{p}$-norm. $\Spsa^{p}(\hilh)$ will denote the subset of self-adjoint operators in $\Sp^{p}(\hilh)$.

\item We denote by $\bh$, $\conth$, and $\unih$ the spaces of bounded, contraction, and unitary operators, respectively.

\item If $T\in\bh$, we let $\sigma(T)$ denote the spectrum of $T$. For any $k\in\N$, we will use the notation $(T)^{k}:=\underbrace{T,\ldots,T}_{\text{k}}$.
	
\item For $\Omega\subset\C$, we denote by $\|\cdot\|_{\infty,\Omega}$ the supremum norm, with the supremum taken over $\Omega$. For $n\in\N$, $C^{n}(\Omega)$ represents the set of functions that are $n$-times continuously differentiable on $\Omega$, while $C(\Omega)$ denotes the space of all continuous functions on $\Omega$. $C_{c}^{n}(\R)$ refers to the subclass of $C^{n}(\R)$ consisting of compactly supported functions. We denote by $C_{0}(\R)$ the space of all continuous functions that vanish at infinity.

\item Denote by $\A(\D)$ the disk-algebra, that is, the space of functions $f : \overline{\D}\rightarrow\C$ analytic on $\D$ and continuous on $\cir$, equipped with $\|\cdot\|_{\infty,\overline{\D}}$. For $f\in\A(\D)$, the notation $f_{\cir}$ stands for the restriction of $f$ on $\cir$. We let $\mathcal{B}(\cir)$ denote the space of all bounded measurable functions on $\cir$.

\item We use the letters $c,d,C$ to denote positive constants, and sometimes add subscripts to indicate their dependencies. For instance, $c_{\alpha}$ denotes a constant depending only on the parameter $\alpha$.

\item Let $H_{0}$ and $V$ be two operators in $\hilh$, then the $n$-th order Taylor remainder with respect to a linear path is defined by $$\mathcal{R}^{\lin}_{n}(f,H_{0},V)=f(H_{0}+V)-f(H_{0})-\sum_{k=1}^{n-1}\frac{1}{k!}\frac{d^{k}}{ds^{k}}\bigg|_{s=0}f(H_{0}+sV)$$
for sufficiently smooth function $f$ and $H_{0},H_{0}+V$ are self-adjoint operators/contractions.

\item Let $T\in\conth$ and $A=A^{*}\in\bh$, then the $n$-th order Taylor remainder with respect to a multiplicative path is defined by $$\mathcal{R}^{\mult}_{n}(f,T,A)=f(e^{iA}T)-f(T)-\sum_{k=1}^{n-1}\frac{1}{k!}\frac{d^{k}}{ds^{k}}\bigg|_{s=0}f(e^{isA}T)$$
for sufficiently smooth function $f$.
	
\item Let $X_{1},\ldots,X_{n},Y$ be Banach spaces. We let $\mathcal{B}_{n}(X_{1}\times\cdots\times X_{n},Y)$ denote the space of bounded $n$-linear operators $T:X_{1}\times\cdots\times X_{n}\to Y$, equipped with the norm
\[\|T\|_{\mathcal{B}_{n}(X_{1}\times\cdots\times X_{n},Y)}:=\sup_{\|x_{i}\|\le1,\,1\le i\le n}\|T(x_{1},\ldots,x_{n})\|.\]
In the case when $X_{1}=\cdots=X_{n}=Y$, we will simply denote this space by $\mathcal{B}_{n}(Y)$.
\end{enumerate}

\section{Preliminaries}\label{Sec2}

\subsection{Multiple operator integrals associated to normal operators} The following definition of multiple operator integration was developed in \cite{CoLemSu21}, which is based on the construction of \cite{Pa69}.

Let $A$ be a (possibly unbounded) normal operator in $\hilh$. Denote by $E^{A}$ the spectral measure of $A$ and by $\sigma(A)$ the spectrum of $A$. Let $e$ be a separating vector of the von Neumann algebra $W^{\ast}(A)$ generated by $A$ (see \cite[Corollary 14.6]{Conway-Book}). Then, by \cite[Theorem 15.3]{Conway-Book}, the measure $\lambda_{A}$ defined by
\[\lambda_{A}(\cdot)=\left\|E^{A}(\cdot)e\right\|^{2}\]
is a positive measure on the Borel subsets of $\sigma(A)$ such that $\lambda_{A}$ and $E^{A}$ both have the same sets of measure zero. The measure $\lambda_{A}$ is called a scalar-valued spectral measure for $A$. According to \cite[Theorem 15.10]{Conway-Book}, we obtain a $w^\ast$-continuous $\ast$-representation
\[f\in L^{\infty}(\lambda_{A})\mapsto f(A):=\int_{\sigma(A)}f(x)\,dE^{A}(x)\in\bh.\]
As a matter of fact, the space $L^{\infty}(\lambda_{A})$ does not depend on the choice of the scalar-valued spectral measure $\lambda_{A}$.

Let $n\in\N$ and let $A_{1},\ldots,A_{n+1}$ be normal operators in $\hilh$ with scalar-valued spectral measures $\lambda_{A_{1}},\ldots,\lambda_{A_{n+1}}$. We let
\[\Gamma^{A_{1},A_{2},\ldots,A_{n+1}}:L^{\infty}(\lambda_{A_{1}})\otimes\cdots\otimes L^{\infty}(\lambda_{A_{n+1}})\to\mathcal{B}_{n}(\Sp^{2}(\hilh))\]
be the unique linear map such that for any $f_{i}\in L^{\infty}(\lambda_{A_{i}})$, $i=1,\ldots,n+1$ and for any $X_{1},\ldots,X_{n}\in\Sp^{2}(\hilh)$,
\begin{align*}
&\left[\Gamma^{A_{1},A_{2},\ldots,A_{n+1}}(f_{1}\otimes\cdots\otimes f_{n+1})\right](X_{1},\ldots,X_{n})\\
&\quad=f_{1}(A_{1})X_{1}f(A_{2})X_{2}\cdots X_{n}f_{n+1}(A_{n+1}).
\end{align*}

We have a natural inclusion $L^{\infty}(\lambda_{A_{1}})\otimes\cdots\otimes L^{\infty}(\lambda_{A_{n+1}})\subset L^{\infty}\left(\prod_{i=1}^{n+1}\lambda_{A_{i}}\right)$ which is $w^{\ast}$-dense. According to \cite[Proposition 3.4 and Corollary 3.9]{CoLemSu21}, $\Gamma^{A_{1},A_{2},\ldots,A_{n+1}}$ extends to a unique $w^{\ast}$-continuous isometry still denoted by
\[\Gamma^{A_{1},A_{2},\ldots,A_{n+1}}:L^{\infty}\left(\prod_{i=1}^{n+1}\lambda_{A_{i}}\right)\to\mathcal{B}_{n}(\Sp^{2}(\hilh)).\]

\begin{dfn}\label{MOI-Def}
For $\phi\in L^{\infty}\left(\prod_{i=1}^{n+1}\lambda_{A_{i}}\right)$, the transformation $\Gamma^{A_{1},A_{2},\ldots,A_{n+1}}(\phi)$ is called a multiple operator integral (in short MOI) associated with $A_{1},\ldots,A_{n+1}$ and the symbol $\phi$.
\end{dfn}

Note that $\mathcal{B}_{n}(\Sp^{2}(\hilh))$ is a dual space, and the $w^{\ast}$-continuity of $\Gamma^{A_{1},A_{2},\ldots,A_{n+1}}$ means that if a net $(\phi_{i})_{i\in I}$ in $L^{\infty}\left(\prod_{i=1}^{n+1}\lambda_{A_{i}}\right)$ converges to $\phi\in L^{\infty}\left(\prod_{i=1}^{n+1}\lambda_{A_{i}}\right)$ in the $w^{\ast}$-topology, then for any $X_{1},\ldots,X_{n}\in\Sp^{2}(\hilh)$
\[\left[\Gamma^{A_{1},A_{2},\ldots,A_{n+1}}(\phi_{i})\right](X_{1},\ldots,X_{n})\mathrel{\mathop{\longrightarrow}^{\mathrm{\text{weakly in $\Sp^{2}(\hilh)$}}}_i}\left[\Gamma^{A_{1},A_{2},\ldots,A_{n+1}}(\phi)\right](X_{1},\ldots,X_{n}).\]

Let $\Omega\supset\prod_{i=1}^{n+1}\sigma(A_{i})$ and $\phi:\Omega\to\C$ be a bounded Borel function. One can define
\[\Gamma^{A_{1},A_{2},\ldots,A_{n+1}}(\phi):\Sp^{2}(\hilh)\times\cdots\times\Sp^{2}(\hilh)\to\Sp^{2}(\hilh)\]
by setting
\[\Gamma^{A_{1},A_{2},\ldots,A_{n+1}}(\phi):=\Gamma^{A_{1},A_{2},\ldots,A_{n+1}}(\widetilde{\phi}),\]
where $\widetilde{\phi}$ is the class of its restriction $\phi|_{\sigma(A_{1})\times\cdots\times\sigma(A_{n+1})}$ in $L^{\infty}\left(\prod_{i=1}^{n+1}\lambda_{A_{i}}\right)$.
\vspace*{0.2cm}

To conclude this subsection, let us mention that there are various definitions of multiple operator integrals in the literature (see, e.g., \cite{BiSo66,DaKr56,Pe06,PoSkSu13}) which are associated to a smaller set of symbols compared the MOI defined in this section. However, connections between these different definitions are explored in \cite[Remark 3.6]{CoLemSu21} and \cite[Proposition 4.3.2]{SkTo19}.

\subsection{$\Sp^{p}$-boundedness of MOI} In this subsection, we recall a few results from \cite{Co22} and \cite{Co24}, which we need in the subsequent sections.

Let $n\in\N$. Consider $1<p,p_{j}<\infty$, for $j=1,\ldots,n$ such that $\frac{1}{p}=\frac{1}{p_{1}}+\cdots+\frac{1}{p_{n}}$. Let $A_{1},\ldots,A_{n+1}$ be normal operators and $\phi\in L^{\infty}\left(\prod_{i=1}^{n+1}\lambda_{A_{i}}\right)$. We will write
\[\Gamma^{A_{1},A_{2},\ldots,A_{n+1}}(\phi)\in \mathcal{B}_{n}(\Sp^{p_{1}}(\hilh)\times\cdots\times\Sp^{p_{n}}(\hilh),\Sp^{p}(\hilh))\]
if the MOI $\Gamma^{A_{1},A_{2},\ldots,A_{n+1}}(\phi)$ defines a bounded $n$-linear mapping
\[\Gamma^{A_{1},A_{2},\ldots,A_{n+1}}(\phi):\left(\Sp^{2}(\hilh)\cap\Sp^{p_{1}}(\hilh)\right)\times\cdots\times\left(\Sp^{2}(\hilh)\cap\Sp^{p_{n}}(\hilh)\right)\to\Sp^{p}(\hilh),\]
where $\Sp^{2}(\hilh)\cap\Sp^{p_{i}}(\hilh)$ is equipped with the $\|\cdot\|_{p_{i}}$-norm. By density of $\Sp^{2}(\hilh)\cap\Sp^{p_{i}}(\hilh)$ in $\Sp^{p_{i}}(\hilh)$, this mapping has an unique extension 
\[\Gamma^{A_{1},A_{2},\ldots,A_{n+1}}(\phi):\Sp^{p_{1}}(\hilh)\times\cdots\times\Sp^{p_{n}}(\hilh)\to\Sp^{p}(\hilh).\]

\subsection{Divided difference}
Consider $\Omega=\R$ or $\cir$ and $n\in\N$. Then for $f:\Omega\to\C$ we define its divided difference $f^{[n]}:\Omega^{n+1}\to\C$ recursively as follows. Let $f$ be $n$-times differentiable on $\Omega$. Then $f^{[n]}:\Omega^{n+1}\to\C$ is defined by
\[f^{[n]}(\lambda_{0},\ldots,\lambda_{n})=\lim_{\lambda\to\lambda_{n}}\frac{f^{[n-1]}(\lambda_{0},\ldots,\lambda_{n-2},\lambda)-f^{[n-1]}(\lambda_{0},\ldots,\lambda_{n-2},\lambda_{n-1})}{\lambda-\lambda_{n-1}}.\]
Note that the zeroth order divided difference of $f$ is simply the function itself, that is, $f^{[0]}=f$. Moreover, it follows from \cite[Section 2.2]{SkTo19}, if $f^{(n)}$ is bounded, then $f^{[n]}$ is also bounded.

In the above (and throughout the article), by the differentiability of $f:\cir\to\C$ at $z_{0}\in\cir$, we mean the limit
\[f'(z_{0}):=\lim_{z\in\cir,\,z\to z_{0}}\frac{f(z)-f(z_{0})}{z-z_{0}},\]
provided that this limit exists.

Finally, we end this section with the following result from \cite[Theorem 2.7]{Co22} and \cite[Theorem 3.3]{Co24}.

\begin{thm}\label{MOIEst-Self+Uni}
Let $n\in\N$ and $f:\Omega~(=\R~\text{or}~\cir)\to\C$ be $n$-times differentiable such that $f^{(n)}$ is bounded. Let $1<p,p_{j}<\infty$, $j=1,\ldots,n$ be such that $\frac{1}{p}=\frac{1}{p_{1}}+\cdots+\frac{1}{p_{n}}$. Let $A_{1},\ldots,A_{n+1}$ be normal (self-adjoint or unitary) operators in $\hilh$. Consider
\begin{enumerate}[{\normalfont(i)}]
\item $\Omega=\R$ and $p_{1}=\cdots=p_{n}=np$ in the case when the operators are self-adjoint;
\item $\Omega=\cir$ if the operators are unitaries.
\end{enumerate}
Then
\begin{align*}
\Gamma^{A_{1},A_{2},\ldots,A_{n+1}}(f^{[n]})\in \mathcal{B}_{n}(\Sp^{p_{1}}(\hilh)\times\cdots\times\Sp^{p_{n}}(\hilh),\Sp^{p}(\hilh))
\end{align*}
and there exists a constant $c_{p,n}>0$ such that
\begin{align*}	\left\|\Gamma^{A_{1},A_{2},\ldots,A_{n+1}}(f^{[n]})\right\|_{\mathcal{B}_{n}(\Sp^{p_{1}}(\hilh)\times\cdots\times\Sp^{p_{n}}(\hilh),\Sp^{p}(\hilh))}\le c_{p,n}\,\|f^{(n)}\|_{\infty,\Omega}.
\end{align*}
In particular, $\Gamma^{A_{1},A_{2},\ldots,A_{n+1}}(f^{[n]})\in\mathcal{B}_{n}(\Sp^{p}(\hilh))$, with
\[\left\|\Gamma^{A_{1},A_{2},\ldots,A_{n+1}}(f^{[n]})\right\|_{\mathcal{B}_{n}(\Sp^{p}(\hilh))}\le c_{p,n}\,\|f^{(n)}\|_{\infty,\Omega}.\]
\end{thm}

As evident from this section, MOIs associated to self-adjoint and unitary operators are well-studied. These results play a crucial role in obtaining various differentiability results (see, e.g., \cite[Theorems 3.1, 3.2]{Co22} and \cite[Theorem 5.1]{Co24}), and we will further use them to establish trace formulas in subsequent sections. However, similar results for contractions have not been as thoroughly explored. Since our article also addresses trace formulas for contractions, which, in turn, requires the $\Sp^{p}$-boundedness for MOIs associated with contractions, the next section addresses this gap.

\section{Differentiability of functions of contractions}\label{DoFoC}

Differentiability of operator functions is an old problem, initiated by Daletskii and Krein in \cite{DaKr56}. This study has often been motivated by problems in perturbation theory. For example, \cite{PoSkSu14} naturally leads to the question of the existence and the representation of the
derivatives of
\[\varphi:t\in[0,1]\mapsto f(S+tV)-f(S)\in\Sp^{p}(\hilh)\]
where $S,S+V\in\conth$,  $V\in\Sp^{p}(\hilh)$ for $1<p<\infty$, and $f\in\A(\D)$. The existence of $\varphi(t)\in\Sp^{p}(\hilh)$ for $t\in[0,1]$ follows from \cite[Theorem 6.4]{KiPoShSu14}.

In \cite{Pe09}, it is proved that if $f,f'\in\A(\D)$, then the first order derivative of $\varphi$ exists in the Hilbert-Schmidt norm. Beyond this, no further results are known. However, for the analogous problem involving self-adjoint and unitary cases, more is known. In fact, optimal results have been obtained, see \cite[Theorem 3.2]{Co22} and \cite[Theorem 5.1]{Co24}. For the solutions of the corresponding problems in the self-adjoint and unitary cases the MOI developed in \cite{CoLemSu21} - as discussed in the previous section - plays a major role, as it is very general.

In this section, we will first extend the definition of MOI discussed in Section \ref{Sec2} to the case of contractions, replacing spectral measures by semi-spectral measures. We refer to \cite{Pe87} for the first occurrence of this construction. Following that we will examine the $n$-th order $\Sp^{p}$-differentiability of $\varphi$ for $1<p<\infty$.

\subsection{Extension of multiple operator integrals in the case of contractions}
	
Let $T\in\conth$, that is, $\|T\|\le1$. Denote by $\mathcal{P}$ the space of polynomials on $\overline{\D}$ equipped with the norm $\|\cdot\|_{\infty,\overline{\D}}$. By von Neumann's inequality, the mapping
\begin{equation}\label{VNmap}
P\in\mathcal{P}\mapsto P(T)
\end{equation}
is a contraction. The space $\mathcal{P}$ is dense in $\A(\D)$ so the map \eqref{VNmap} uniquely extends to a contraction
\begin{equation*}
f\in\A(\D)\mapsto f(T).
\end{equation*} 
Note that by the maximum principle, we have, for any $f\in\A(\D)$,
\begin{equation*}
\|f(T)\|\le\|f\|_{\infty,\cir}.
\end{equation*}
	
By the Sz.-Nagy dilation theorem, there exists a (separable) Hilbert space $\hilk$ and a unitary $U$ on $\hilk$ such that $\hilh\subset\hilk$ and
\begin{equation}\label{Compression-1}
T^n=P_{\hilh}U^{n}|_{\hilh}\quad\forall~n\ge 0,
\end{equation}
where $P_{\hilh}$ is the orthogonal projection from $\hilk$ onto $\hilh$. Such a $U$ is called a unitary dilation of $T$.
	
Let $E$ be the spectral measure of $U$. For any Borel subset $\Delta\subset\cir$, we define $\mathcal{E}(\Delta)\in \mathcal{B}(\hilh)$ by
\begin{equation}\label{Compression-2}
\mathcal{E}(\Delta)=P_{\hilh} E(\Delta)|_{\hilh}.
\end{equation}
We say that $\mathcal{E}$ is the compression of the spectral measure $E$. The mapping $\mathcal{E}$ is called a semi-spectral measure in the following sense.

\begin{dfn}
Let $(X,\mathcal{A})$ be a measurable space. A map
$\mathcal{E}:\mathcal{A}\rightarrow\mathcal{B}(\hilh)$
is called a semi-spectral measure if
\begin{enumerate}[{\normalfont(i)}]
\item $\mathcal{E}(\varnothing)=0, \mathcal{E}(X)=I$;
\item For any $\Delta\in\mathcal{A}, \mathcal{E}(\Delta)\ge 0$;
\item For any sequence $\{\Delta_{k}\}_{k\ge 1}$ of pairwise disjoint elements of $\mathcal{A}$,
$$\mathcal{E}\left(\displaystyle\bigcup_{k=1}^{\infty}\Delta_k\right)=\sum_{k=1}^{\infty}\mathcal{E}(\Delta_k),$$
where the series converges in the strong operator topology (SOT).
\end{enumerate}
\end{dfn}
It is proved in \cite{Na40}  that any semi-spectral measure is the compression of a spectral measure. Let $(X,\mathcal{A})$ be a measurable space, and let $f:X\rightarrow\C$ be a bounded continuous function. For such $f$ we can define	
$$\int_{X}f(x)\
\text{d}\mathcal{E}(x)$$
as the limit of the sums
$$\sum f(x_{k})\mathcal{E}(\Delta_{k}),~x_{k} \in\Delta_{k},$$
over all finite measurable partitions $\{\Delta_{k}\}_{k}\subset\mathcal{A}$ of $X$.
	
Let $U\in\mathcal{B}(\hilk)$ be a unitary dilation of $T\in\conth$. We denote by $E_{U}$ the spectral measure of $U$ and by $\mathcal{E}_{T}$ the compression of $E_{U}$ on $\mathcal{H}$. $\mathcal{E}_T$ is called a semi-spectral measure of $T$. It follows easily from \eqref{Compression-1} and \eqref{Compression-2} that for any $f\in\A(\D)$,
\begin{equation}\label{Compression-3}
f(T)=\int_{\cir}f(\xi)\ \text{d}\mathcal{E}_{T}(\xi)=P_{\hilh} \left( \int_{\cir}f(\xi)\ \text{d}E_U(\xi)\right)\bigg|_{\hilh}= P_\hilh f(U)|_\hilh.
\end{equation}
		
It is easy to see that, from \eqref{Compression-3}, $f(T)$ is independent of the choice of unitary dilation $U$.   So, from now onwards, we fix a special choice of unitary dilation, namely the Sch\"affer matrix unitary dilation (see below).
	
Given a  contraction $T\in\bh$, Sch\"affer found an explicit power unitary dilation $U_T$ of a contraction $T$. Let $\ell_2(\hilh)=\oplus_{1}^{\infty}\hilh$. The dilation space is $\hilk=\ell_{2}(\hilh)\oplus \hilh\oplus\ell_{2}(\hilh),$ which is separable and the unitary operator is given by
\begin{align}\label{Schaffer-Dil}
U_{T}=\begin{blockarray}{ccccc}
\ell_2(\hilh)&\hilh&\ell_2(\hilh)& \\[4pt]
\begin{block}{[ccc]cc}
S^*&0&0&\ell_{2}(\hilh) \\[3pt]
D_{T^*}\widetilde{P}_{\hilh}&T&0&\hilh\\[3pt]
-T^{*}\widetilde{P}_{\hilh}& D_{T}&S&\ell_{2}(\hilh)\\[3pt]
\end{block}
\end{blockarray},
\end{align}
where $S$ is the unilateral shift on $\ell_{2}(\hilh)$ of multiplicity $\dim(\hilh)$, $\widetilde{P}_{\hilh}$ is the orthogonal projection from $\ell_{2}(\hilh)$ onto $\hilh\oplus 0\oplus0\oplus\cdots$, $D_{T}=(I-T^{*}T)^{1/2}$ and $D_{T^*}=(I-TT^{*})^{1/2}$. 
	
From now on, by the unitary dilation $U$ of $T$, we always always refer to the Sch\"affer matrix unitary dilation $U_T$. Note that such a dilation does not necessarily have to be minimal. However, the advantage of this dilation is that it allows us to consider unitary dilations of contractions on $\hilh$ on a same Hilbert space $\hilk=\ell_{2}(\hilh)\oplus\hilh\oplus\ell_{2}(\hilh)$.

We will now explain how to define multiple operator integrals associated to $n$ contractions $T_{1},T_{2},\ldots,T_{n}\in\conth$ and to a function in $\A(\D^{n})$.

For any $1\le k\le n$, let $U_{k}\in \mathcal{B}(\hilk)$ be the unitary dilation of $T_{k}$. Let $f=f_{1}\otimes\cdots\otimes f_{n}$, where $f_{k}\in\A(\D)$. Let $X_{1},\ldots,X_{n-1}\in\Sp^{2}(\hilh)$. We set
$$\Gamma(f)=f_1(T_1) X_1 f_2(T_2)X_{2}\cdots X_{n-2}f_{n-1}(T_{n-1}) X_{n-1}f_n(T_n).$$
By linearity, we can define $\Gamma(f)$ for any $f\in\A(\D)\otimes\cdots\otimes\A(\D)$. By \eqref{Compression-3}, it is easy to see that
\begin{equation}\label{MOI-Com-1}
\Gamma(f)=P_{\hilh}\left[\Gamma^{U_{1},\ldots,U_{n}}(f)\right](\widetilde{X}_{1},\ldots,\widetilde{X}_{n-1})\big|_{\hilh},
\end{equation}
where $\widetilde{X}_{k}=P_{\hilh}X_{k}P_{\hilh}$, $1\le k\le n-1$ and $P_\hilh:\hilk\to\hilh$ is the projection onto $\hilh$. In particular (by \cite[Proposition 3.4]{CoLemSu21}),
\begin{equation}\label{MOI-Ineq}
\|\Gamma(f)\|_{2}\le\|f\|_{\infty,\overline{\D}^n}\|X_1\|_{2}\ldots\|X_{n-1}\|_{2},
\end{equation}
which shows that $\Gamma(f)$ is well defined, that is, its definition does not depend on the representation of $f$ in $\A(\D)\otimes\cdots\otimes\A(\D)$. Note that, as in the case of a single contraction, $\Gamma(f)$ is independent of the choice of unitary dilations. Moreover, the well-defined map
$$f\in\A(\D)\otimes\cdots\otimes\A(\D)\mapsto\Gamma(f)$$
is linear and continuous, and hence extends uniquely to a continuous linear map on the $\|\cdot\|_{\infty,\overline{\D}^n}$-closure of $\A(\D)\otimes\cdots\otimes\A(\D)$, which we denote as $\A(\D^n)$. For $f\in\A(\D^n)$, the resulting operator will be denoted by
$$\left[\Gamma^{T_1,\ldots,T_n}(f)\right](X_1,\ldots,X_{n-1}),$$
and by \eqref{MOI-Ineq}, the element
$$\Gamma^{T_1,\ldots,T_n}(f)$$
belongs to $\mathcal{B}_{n-1}(\Sp^{2}(\hilh))$ and has a norm less than $\|f\|_{\infty,\overline{\D}^n}$. $\Gamma^{T_1,\ldots,T_n}(f)$ will be called multiple operator integral associated to $T_1,\ldots,T_n$ and $f$.
	
Finally, note that \eqref{MOI-Com-1} holds true for $\Gamma^{T_1,\ldots,T_n}(f)$, that is, for any $X_1,\ldots,X_{n-1}\in\Sp^{2}(\hilh)$, we have
\begin{equation}\label{MOI-Cont-Def2}
\left[\Gamma^{T_1,\ldots,T_n}(f)\right](X_1,\ldots,X_{n-1})=P_{\hilh}\left[\Gamma^{U_1,\ldots,U_n}(f)\right](\widetilde{X}_1,\ldots,\widetilde{X}_{n-1})|_{\hilh}.
\end{equation}

Going further we need the following crucial lemma.

\begin{lma}\label{Lem-Div-Discalg}
Let $n\in\N$, and let $f \in\A(\D)$ be such that, for any $1\le k\le n, f^{(k)}\in\A(\D)$. Then, for any $1\le k \le n$, $f^{[k]}\in\A(\D^{k+1})$.
\end{lma}
	
\begin{proof}
Let $1\le k\le n$. For any $r\in(0,1)$, write $\cir_{r}=\left\lbrace z\in\C\ | \ |z|=r \right\rbrace$ and for a function $g\in\A(\D)$, we let  $g_r:\cir\rightarrow\C$ to be defined by $g_r(z)=g(rz), z\in\cir$. It is easy to check that for any $(z_0,\ldots,z_k)\in \cir^{k+1}$,
$$g^{[k]}(rz_{0},\ldots,rz_{k})=\dfrac{1}{r^{k}} g_{r}^{[k]}(z_{0},\ldots,z_{k}).$$
Hence, by analyticity and the maximum principle,
$$\|g^{[k]}\|_{\infty, \overline{\D}^{k+1}} = \underset{r\to 1}{\lim}~\|g^{[k]}\|_{\infty, \cir_{r}^{k+1}}\le\underset{r\to 1}{\lim}~\dfrac{1}{r^k}\|g_{r}^{[k]}\|_{\infty,\cir^{k+1}}.$$
By \cite[Lemma 3.2]{Sk17Adv}, there exists a constant $c_k$ depending only on $k$ such that
$$\|g_{r}^{[k]}\|_{\infty,\cir^{k+1}}\le c_{k}\|g_{r}^{(k)}\|_{\infty,\cir}=c_{k}r^{k}\|g^{(k)}\|_{\infty,\cir_r}.$$
Again, by the maximum principle, we get
\begin{equation}\label{Lem-Div-Discalg-R1}
\|g^{[k]}\|_{\infty,\overline{\D}^{k+1}}\le c_{k} \|g^{(k)}\|_{\infty,\cir}.
\end{equation}
Now, let $\{P_{j}\}_{j\ge1}$ be a sequence of polynomials converging uniformly to $f^{(k)}$ on $\overline{\D}$ and let $\{Q_{j}\}_{j\ge1}$ be a sequence of polynomials such that, for any $j, Q_{j}^{(k)}=P_j$. By \eqref{Lem-Div-Discalg-R1} we have
\begin{align*}
\|f^{[k]}-Q_{j}^{[k]}\|_{\infty,\overline{\D}^{k+1}}&=\|(f-Q_j)^{[k]}\|_{\infty,\overline{\D}^{k+1}} \\
&\le c_{k}\|f^{(k)}-P_{j}\|_{\infty,\cir}\underset{j\to+\infty}{\longrightarrow}0,
\end{align*}
which concludes the proof of the lemma.
\end{proof}

\subsubsection{Boundedness of MOI associated to contractions}

\begin{thm}\label{MOIEst-Cont}
Let $n\in\N$ and $f\in\A(\D)$ be such that $f^{(k)}\in\A(\D)$ for $1\le k\le n$. Let $1<p,p_{j}<\infty$, $j=1,\ldots,n$ be such that $\frac{1}{p}=\frac{1}{p_{1}}+\cdots+\frac{1}{p_{n}}$. Let $T_1,\ldots,T_{n+1}\in\conth$. Then
$$\Gamma^{T_1,\ldots,T_{n+1}}(f^{[n]})\in\mathcal{B}_{n}(\Sp^{p_{1}}(\hilh)\times\cdots\times\Sp^{p_{n}}(\hilh),\Sp^p(\hilh))$$
and there exists a constant $c_{p,n}>0$ such that for any $K_{i}\in\Sp^{p_{i}}(\hilh),~1\le i\le n$,
\begin{align*}
\left\|\left[\Gamma^{T_1,\ldots,T_{n+1}}(f^{[n]})\right](K_{1},\ldots,K_{n})\right\|_{p}\le c_{p,n}\left\|f^{(n)}\right\|_{\infty,\overline{\D}}\|K_{1}\|_{p_{1}}\cdots\|K_{n}\|_{p_{n}}.
\end{align*}
In particular, $\Gamma^{T_1,\ldots,T_{n+1}}(f^{[n]})\in\mathcal{B}_{n}(\Sp^{p}(\hilh))$, with
\begin{align*}
\left\|\Gamma^{T_1,\ldots,T_{n+1}}(f^{[n]})\right\|_{\mathcal{B}_{n}(\Sp^{p}(\hilh))}\le c_{p,n}\left\|f^{(n)}\right\|_{\infty,\overline{\D}}.
\end{align*}
\end{thm}

\begin{proof} The assumptions on $f$ and Lemma \ref{Lem-Div-Discalg} ensures that $f^{[n]}\in\A(\D^{n+1})$. The proof now trivially follows from \eqref{MOI-Cont-Def2} and Theorem \ref{MOIEst-Self+Uni}.
\end{proof}

Next, we explore higher order derivatives of functions of contractions.

\subsubsection{Auxiliary lemmas}
The differentiability result is established based on the following two lemmas.

\begin{ppsn}\label{Perturbation-Formula}
Let $1<p<\infty$ and $n\in\N$. Let $T_1,\ldots,T_{n-1},S,T\in\conth$ and assume that $S-T\in\Sp^p(\hilh)$. Let $f\in\A(\D)$ be such that $f^{(k)}\in\A(\D)$ for $1\le k\le n$. Then, for all $K_1,\ldots,K_{n-1}\in\Sp^p(\hilh)$ and for any $1\le i\le n$ we have
\begin{align}
\nonumber&\left[\Gamma^{T_1,\ldots,T_{i-1},S,T_i,\ldots,T_{n-1}}(f^{[n-1]})-\Gamma^{T_1,\ldots,T_{i-1},T,T_i,\ldots,T_{n-1}}(f^{[n-1]})\right](K_1,\ldots,K_{n-1})\\
\label{Perturbation-Formula-Eq}&\hspace*{0.5cm}=\left[\Gamma^{T_1,\ldots,T_{i-1},S,T,T_i,\ldots,T_{n-1}}(f^{[n]})\right](K_1,\ldots,K_{i-1},S-T,K_i,\ldots,K_{n-1}).
\end{align}
\end{ppsn}

\begin{proof} 
The proof is based on the approximation of $f\in\A(\D)$. It is well-known that $f_{\cir}:\cir\to\C$ is a continuous function with vanishing negative Fourier coefficients. Consider the $k$-th Ces\`aro sum
\begin{align}\label{Cesaro-Seqn}
\varphi_{k}(z)=\frac{1}{k}\sum_{i=0}^{k-1}S_{i}(f)(z),~z\in\overline{\D},
\end{align}
where $S_{i}(f)(z)$	is the $i$-th partial sum of the Taylor series of $f$. Since $f_{\cir}^{(m)}\in C(\cir)$ for $1\le m\le n$, we have
\begin{align*}
\left\|(\varphi_k)^{(m)}-f^{(m)}\right\|_{\infty,\cir}\underset{k\to+\infty}{\longrightarrow}0\quad\text{for}~1\le m\le n.
\end{align*}
By \eqref{Lem-Div-Discalg-R1}, this further implies that
\begin{align*}
\left\|(\varphi_k)^{[m]}-f^{[m]}\right\|_{\infty,\overline{\D}^{m+1}}\le c_{m}\left\|(\varphi_k)^{(m)}-f^{(m)}\right\|_{\infty,\cir}\underset{k\to+\infty}{\longrightarrow}0\quad\text{for}~1\le m\le n.
\end{align*} 
Therefore, according to Theorem \ref{MOIEst-Cont}, it is enough to prove the formula for $\varphi_k$ instead of $f$ and the proof now follows from the same computations than those performed in the first part of the proof of \cite[Proposition 2.5]{ChCoGiPr24}. 

For $n=1$, the perturbation formula means that, if $S,T\in\conth$ are such that $S-T\in\Sp^{p}(\hilh)$, then
\[f(S)-f(T)=\left[\Gamma^{S,T}(f^{[1]})\right](S-T).\]
\end{proof} 

Thanks to Proposition \ref{Perturbation-Formula}, the proof of the next lemma follows the same line of argument as of \cite[Lemma 3.7]{ChCoGiPr24}, hence we omit it.

\begin{lma}\label{Lemmadiff}
Let $1<p<\infty$ and $n\in\N$. Let $S_1,\ldots,S_{n+1},T_1,\ldots,T_{n+1}\in\conth$ and $K_{1},\ldots,K_{n}\in\Sp^{np}(\hilh)$. Let $f\in\A(\D)$ be such that $f^{(k)}\in\A(\D)$ for $1\le k\le n$. Then, for every $\epsilon>0$, there exists a constant $\delta>0$ such that, if $\|S_{i}-T_{i}\|\le\delta$ for every $1\le i\le n+1$, then
$$\left\|\left[\Gamma^{S_1,\ldots, S_{n+1}}(f^{[n]})- \Gamma^{T_1,\ldots, T_{n+1}}(f^{[n]})\right](K_1,\ldots,K_n) \right\|_{p}\le\epsilon \|K_1\|_{np} \cdots\|K_n\|_{np}.$$
\end{lma}

\subsection{Differentiability result}

Let $T\in\conth$ and $K\in\Sp^p(\hilh)$ for some $1<p<\infty$ such that $T+K\in\conth$. By \cite[Theorem 6.4]{KiPoShSu14}, it follows that if $f\in\A(\D)$, then 
\[f(T+tK)-f(T)\in\Sp^{p}(\hilh), \quad\forall~ t\,\in[0,1].\]
Therefore, we can consider G\^ateaux differentiability of the operator function as follows.

\begin{thm}\label{Differentiation}
Let $1<p<\infty$ and $f\in\A(\D)$ be such that $f^{(k)}\in\A(\D)$ for $1\le k\le n$. Let $T\in\conth$ and $K\in\Sp^{p}(\hilh)$ be such that $T+K\in\conth$. We consider the function
$$\varphi: t\in[0,1]\mapsto f(T+tK)-f(T)\in\Sp^p(\hilh).$$
Then $\varphi$ is G\^ateaux differentiable on $[0,1]$ and for every $t\in[0,1]$ and every integer $1\le k\le n$,
\begin{align*}
\frac{1}{k!}\varphi^{(k)}(t)=\left[\Gamma^{T+tK,\ldots,T+tK}(f^{[k]})\right](K,\ldots,K).
\end{align*}
\end{thm}

\begin{proof}
We prove the claim by induction on $k$, $1\le k\le n$. Since the argument for the base case is similar to the induction step, we omit it. Hence, let $1\le k< n$ and assume that for $1\le j\le k$ 
\begin{align*}
\varphi^{(j)}(t)=j!\left[\Gamma^{T+tK,\ldots,T+tK}(f^{[j]})\right](K,\ldots,K).
\end{align*}
Let $s,t\in[0,1]$, $s\neq t$. By Proposition \ref{Perturbation-Formula} we have
\begin{align*}
&\frac{\varphi^{(j)}(s)-\varphi^{(j)}(t)}{s-t}\\
&=\frac{j!}{s-t}\left[\Gamma^{T+sK,\ldots,T+sK}(f^{[j]})-\Gamma^{T+tK,\ldots,T+tK}(f^{[j]})\right](K,\ldots,K) \\
&=j! \sum_{i=1}^{j+1}\left[\Gamma^{(T+sK)^{i},(T+tK)^{j+2-i}}(f^{[j+1]})\right](K,\ldots,K).
\end{align*}
According to Lemma \ref{Lemmadiff}, the latter converges, as $s\to t$, to
\begin{align*}
j!\sum_{i=1}^{j+1}\left[\Gamma^{T+tK,\ldots,T+tK}(f^{[j+1]})\right](K,\ldots,K)=(j+1)!\left[\Gamma^{T+tK,\ldots,T+tK}(f^{[j+1]})\right](K,\ldots,K),
\end{align*}
which concludes the proof.
\end{proof}

The preceding result generalizes \cite[Theorem 6.1]{Pe09}, where only the case $n=1$ was proved under the assumption $K\in\mathcal{S}^{2}(\hilh)$. Finally, as a consequence of Theorem \ref{Differentiation}, we derive the following estimate for operator Taylor remainder.

\begin{thm}
Let $n\in\N$ and $1<n<p<\infty$. Let $T,T+K\in\conth$, where $K\in\Sp^{p}(\hilh)$. Let $f\in\A(\D)$ be such that $f^{(k)}\in\A(\D)$ for $1\le k\le n$. Then,
\begin{align}
\label{Taylor-Est}\left\|\mathcal{R}_{n}^{\lin}(f,T,K)\right\|_{p/n}\le c_{p,n}\left\|f^{(n)}\right\|_{\infty,\overline{\D}}\|K\|_{p}^{n},
\end{align}
for some constant $c_{p,n}>0$.
\end{thm}

\begin{proof}
The existence of derivatives $\frac{d^{k}}{ds^{k}}\big|_{s=0}f(T+sK)$ for $1\le k\le n-1$ is justified by Theorem \ref{Differentiation}. By Proposition \ref{Perturbation-Formula} and induction on $n$ we obtain
\begin{align}
\nonumber\mathcal{R}_{n}^{\lin}(f,T,K)&=\left[\Gamma^{T+K,(T)^{n-1}}(f^{[n-1]})-\Gamma^{(T)^{n}}(f^{[n-1]})\right](K,\ldots,K)\\
\label{Taylor-Exp-Cont}&=\left[\Gamma^{T+K,(T)^{n}}(f^{[n]})\right](K,\ldots,K).
\end{align}
Lastly, the estimate \eqref{Taylor-Est} follows from \eqref{Taylor-Exp-Cont} and Theorem \ref{MOIEst-Cont}.
\end{proof}

\section{Second order trace formulas}\label{Second order trace formulas}

This and the forthcoming sections serve as applications of the main results obtained in \cite{Co22}, \cite{Co24}, and Section \ref{DoFoC}. The main results discussed there lack specific applications.
In this section, we revisit the second order trace formulas. More precisely, we establish the existence of second order spectral shift function (in short SSF) for the most general class of functions, which unifies all existing results in the literature.

\subsection{The main assumption on the second order divided difference}\label{sectionmainassumption}

To prove the various trace formulas of this section, we will need a certain assumption on the divided difference of second order $f^{[2]}$ of a function $f$. At the end of this subsection, we note that this assumption is very general and includes all existing classes for which the associated trace formulas are valid.

First of all, we need to recall the following results.

\begin{thm}\label{caracS1}
Let $A,B,C$ be normal operators in $\hilh$ and let $\varphi\in L^{\infty}(\lambda_{A}\times\lambda_{B})$ and $\phi\in L^{\infty}(\lambda_A\times \lambda_B \times \lambda_C)$.
\begin{enumerate}[{\normalfont(A)}]
\item{\normalfont\cite[Theorem 1]{Pe85}} The following are equivalent:
\begin{enumerate}[{\normalfont(i)}]
\item $\Gamma^{A,B}(\varphi)$ extends to a bounded linear map from $\mathcal{S}^1(\hilh)$ into itself.
\item There exist a separable Hilbert space $\mathfrak{H}$ and two functions
$$a\in L^{\infty}(\lambda_A;\mathfrak{H}) \quad \text{and} \quad b\in L^{\infty}(\lambda_B;\mathfrak{H})$$
such that
\begin{equation}\label{facto1}
\varphi(t_1,t_2)=\langle a(t_{1}),b(t_{2})\rangle
\end{equation}
for a.e. $(t_1,t_2)\in \sigma(A)\times \sigma(B)$.
\end{enumerate}
\item{\normalfont\cite[Theorem 6.2]{CoLemSu21}} The following are equivalent:
\begin{enumerate}[{\normalfont(i)}]
\item $\Gamma^{A,B,C}(\phi) : \mathcal{S}^2(\hilh) \times \mathcal{S}^2(\hilh) \to \mathcal{S}^1(\hilh)$ is bilinear bounded.
\item There exist a separable Hilbert space $\mathfrak{H}$ and two functions
$$a\in L^{\infty}(\lambda_A \times \lambda_B;\mathfrak{H}) \quad \text{and} \quad b\in L^{\infty}(\lambda_B \times \lambda_C;\mathfrak{H})$$
such that
\begin{equation}\label{facto1bil}
\phi(t_1,t_2,t_3)=\langle a(t_{1},t_{2}),b(t_{2},t_{3})\rangle
\end{equation}
for a.e. $(t_1,t_2,t_3)\in \sigma(A)\times \sigma(B)\times \sigma(C)$.
\end{enumerate}
\end{enumerate}
\end{thm}

\begin{rmrk}\label{caracS1rmrk}
In particular if $\varphi(t_1,t_2)=g(t_2)$, where $g:\sigma(B)\to\C$ is Borel measurable and bounded, it follows from the definition of multiple operator integrals that $\Gamma^{A,B}(g):\Sp^{1}(\hilh) \to\Sp^1(\hilh)$ is bounded and that for every $X\in\Sp^1(\hilh)$,
$$\left[\Gamma^{A,B}(g)\right](X)=Xg(B).$$
\end{rmrk}

The implication (ii) $\implies$ (i) of Theorem \ref{caracS1}(B) is the easy one and we refer to the first part of the proof of \cite[Theorem 6.2]{CoLemSu21} for a simple argument. For our purpose, we will make the assumption \eqref{facto1bil} for the function $\phi=f^{[2]}$. As we shall see in the different proofs, this will ensure that the relevant Taylor remainder belongs to $\mathcal{S}^1(\hilh)$ (this is due to its representation as a linear combination of MOI) so that its trace is a well defined element of $\mathbb{C}$. The operators $A,B,C$ will sometimes vary along the proofs, thus, for a twice differentiable $f$, we will need the following convenient assumption on $f^{[2]}$ which will provide a factorization as in \eqref{facto1bil} for every (unitary or self-adjoint) operators $A,B,C$.

\begin{hyp}\label{facto2}
There exists a separable Hilbert space $\mathfrak{H}$ and two bounded Borel functions
$$a:X^2\to  \mathfrak{H} \quad\text{and}\quad b:X^2\to \mathfrak{H}$$
such that
\begin{equation}\label{facto2bis}
f^{[2]}(t_1,t_2,t_3)=\langle a(t_{1},t_{2}),b(t_{2},t_{3})\rangle
\end{equation}
for every $(t_1,t_2,t_3)\in X^3$, where
\begin{enumerate}[{\normalfont(i)}]
\item $X=\cir$ in the case when the operators are contractions or unitaries;
\item $X=\R$ if the operators are self-adjoint.
\end{enumerate}
\end{hyp}

As mentioned in the introduction, we now give two classes of functions that satisfy Hypothesis \ref{facto2}: the Besov classes $B^{2}_{\infty1}(\R)$ and $B^{2}_{\infty1}(\cir)$. For instance, according to \cite[Proposition 5.4]{CoLemSkSu19}, functions in $B^{2}_{\infty1}(\R)$ satisfy property \eqref{facto2bis}. We refer, e.g., to \cite{Pe05} for characterizations of elements of Besov spaces.

The following result was proved for functions in $B^{2}_{\infty1}(\R)$ in \cite[Proposition 5.4]{CoLemSkSu19}, and it was observed that a similar proof works for the following cases as well. Therefore, we omit the proof.

\begin{ppsn}\label{Suff-Ppsn}
The following inclusions hold.
\begin{enumerate}[{\normalfont(i)}]
\item If $f\in B^{2}_{\infty1}(\cir)$, then $f^{[2]}$ satisfies Hypothesis \ref{facto2}.
\item Let $f:\cir\to\C$ be a twice differentiable function with bounded $f''$, and suppose that $f^{[2]}$ belongs to the integral projective tensor product $\mathcal{B}(\cir)\widehat{\otimes}_{i}\mathcal{B}(\cir)\widehat{\otimes}_{i}\mathcal{B}(\cir)$. Then, $f^{[2]}$ satisfies Hypothesis \ref{facto2}.
\end{enumerate}
\end{ppsn}

\begin{ppsn}\label{Example-Ppsn}
Let $g:\R\to\C$ be an operator-Lipschitz function and let $f$ be a primitive of $g$ on $\R$. Then $f^{[2]}$ satisfies Hypothesis \ref{facto2}.
\end{ppsn}

\begin{proof}
First of all, a simple change of variable yields
$$\forall\,(t_1,t_2)\in \R^2, \ f^{[1]}(t_1,t_2) =\int_0^1 g(\lambda t_1+(1-\lambda)t_2)d\lambda.$$
Next, as recalled in the Introduction, there exist a Hilbert space $\cL$ and bounded weakly continuous functions $a,b:\R\to\cL$ such that
\begin{equation*}
\forall\,(s,t)\in\R^2, \ g^{[1]}(s,t)=\left\langle a(s),b(t) \right\rangle.
\end{equation*}
Hence, for every $(t_1,t_2,t_3)\in \mathbb{R}^3$,
\begin{align*}
f^{[2]}(t_1,t_2,t_3)
&=\frac{f^{[1]}(t_1,t_2)-f^{[1]}(t_3,t_2)}{t_1-t_3}\\
&=\int_0^1 \frac{g(\lambda t_1 + (1-\lambda) t_2)-g(\lambda t_3 + (1-\lambda) t_2)}{t_1-t_3}d\lambda \\
&=\int_0^1 g^{[1]}(\lambda t_1 + (1-\lambda) t_2,\lambda t_3 + (1-\lambda) t_2) \lambda d\lambda \\
&=\int_0^1 \left\langle a(\lambda t_1+(1-\lambda) t_2), b(\lambda t_3+(1-\lambda) t_2) \right\rangle \lambda d\lambda.
\end{align*}
Let $\mathfrak{H}:= L^2([0,1],\lambda d\lambda; \cL)$ be the (separable) Hilbert space defined as the Bochner space of (strongly) measurable functions $h : [0,1] \to \cL$ such that $\int_0^1 \| h(\lambda)\|^2 \lambda d\lambda < +\infty$. Then, define $\alpha,\beta:\R^2\to\mathfrak{H}$ as follows: for every $(s,t)\in\R^2$,
$$\alpha(s,t)=\left[\lambda \mapsto a(\lambda s + (1-\lambda) t) \right]\quad \text{and} \quad \beta(s,t)=\left[\lambda \mapsto b(\lambda t+(1-\lambda) s) \right].$$
Then $\alpha, \beta$ are bounded Borel functions and for every $(t_1,t_2,t_3)\in \mathbb{R}^3$,
$$f^{[2]}(t_1,t_2,t_3) = \left\langle \alpha(t_1,t_2), \beta(t_2,t_3) \right\rangle_{\mathfrak{H}},$$
so that $f^{[2]}$ satisfies Hypothesis \ref{facto2}.
\end{proof}

\begin{xmpl} 
Let $g(x)=x^2\sin(1/x)$ for $x\neq 0$ and $g(0)=0$. According to \cite[Example 7]{AlPe16Survey}, $g$ is operator-Lipschitz on $\R$. Hence, by the above Proposition \ref{Example-Ppsn}, the function $f$ defined by $f(x)=\int_0^x g(t)dt$ is a twice differentiable function on $\R$ but not $C^2$ and which satisfies Hypothesis \ref{facto2}. This yields an example of a function for which \cite[Theorem 5.1]{CoLemSkSu19} cannot be applied but whose case is covered by Theorem \ref{SOT-Sacase} in the next section.
\end{xmpl}

\subsection{Koplienko trace formula} 
We first establish the existence of second order SSF for a pair of self-adjoint operators $\{H_{0}+V,H_{0}\}$, where $V\in\Sp^{2}_{sa}(\hilh)$, see Theorem \ref{SOT-Sacase}. We prove the result for functions on $\R$ that are twice differentiable and satisfy \eqref{facto2bis}. This relaxes the continuity of $f''$ assumed in \cite[Theorem 5.1]{CoLemSkSu19}. As a by-product, our method yields a direct proof of the non-negativity of the SSF, which is a well known result. Our goal is achieved with the help of the following lemma.

\begin{lma}\label{Ko-Lem1}
Let $f:\R\to\C$ be twice differentiable and such that there exists a separable Hilbert space $\mathfrak{H}$ and two bounded Borel functions $a,b:\R^{2}\to\mathfrak{H}$ such that
\begin{align*}
f^{[2]}(t_{1},t_{2},t_{3})=\langle a(t_{1},t_{2}),b(t_{2},t_{3})\rangle\quad\forall\,(t_{1},t_{2},t_{3})\in\R^3.
\end{align*}
Let $H_{0}$ be a self-adjoint operator in $\hilh$, and $V\in \Sp^2_{sa}(\hilh)$. Assume either $H_{0}$ is bounded or $f'$ is bounded. Then
$$\mathcal{R}^{\lin}_{2}(f,H_{0},V)\in\Sp^1(\hilh),$$
where the derivative $\frac{d}{ds}\big|_{s=0}f(H_{0}+sV)$ is taken in the $\Sp^2(\hilh)$. Moreover,
\begin{align*}
&\nonumber \Tr(\mathcal{R}^{\lin}_{2}(f,H_{0},V))\\
&=2\int_{0}^{1} (1-t)\Tr\left\{\left[\Gamma^{H_{0}+tV,H_{0}+tV}(\psi_{f})\right](V)\cdot V\right\}dt,
\end{align*}
where $\psi_{f}(t_{1},t_{2})=f^{[2]}(t_{1},t_{2},t_{1})$.
\end{lma}

\begin{proof} 
The assumption on $f^{[2]}$ implies that $f''$ is bounded. Hence, according to \cite[Theorem 3.1 and Theorem 3.2]{Co22}, the function defined by
\begin{equation*}
\varphi:t\in\R\mapsto f(H_0+tV)-f(H_0) \in \mathcal{S}^2(\mathcal{H})
\end{equation*}
is twice differentiable on $[0,1]$ and
\begin{align}
\label{Ko-Lem1-R1}\varphi''(t)=2\left[\Gamma^{H_0+tV,H_0+tV,H_0+tV}(f^{[2]})\right](V,V)\in\Sp^2(\hilh)
\end{align}
and 
\begin{align*}
\mathcal{R}^{\lin}_{2}(f,H_{0},V)=\big[\Gamma^{H_0+V,H_0,H_0}(f^{[2]})\big](V,V).
\end{align*}
Indeed, in order to apply \cite[Theorem 3.1]{Co22}, one only needs to ensure that $f'$ is bounded on a segment containing the union of the spectrum of $H_0+tV$, $t\in [0,1]$, which is the case from the assumptions satisfied by $f$.

Note that, as a function valued in $\Sp^2(\hilh)$, the function $\varphi''$ is strongly measurable. This follows from \cite[II. 1. Theorem 2]{DiUhl77} since $\varphi''$ is valued in a separable Banach space and it is weakly measurable because it is a derivative. Moreover, $\varphi'':\R\to\Sp^2(\hilh)$ is bounded. Hence, it is integrable and
$$\mathcal{R}^{\lin}_{2}(f,H_{0},V)=\int_{0}^{1}(1-t)\varphi''(t)dt,$$
where the integral exists in $\|\cdot\|_{2}$-norm.
Consequently, from the last part of the proof of \cite[Lemma 5.2]{CoLemSkSu19}, it follows that
$$\Tr(\mathcal{R}^{\lin}_{2}(f,H_{0},V))=\int_0^1 (1-t)\Tr(\varphi''(t))dt.$$
Finally, according to \eqref{Ko-Lem1-R1} and \cite[Proposition 2.5]{CoLemSkSu19}, we have
\begin{align*}
\Tr\left(\mathcal{R}^{\lin}_{2}(f,H_{0},V)\right)&=2\int_{0}^{1} (1-t)\Tr\left\{\left[ \Gamma^{H_{0}+tV,H_{0}+tV,H_{0}+tV}(f^{[2]})\right](V,V)\right\}dt\\
&=2\int_0^1(1-t)\Tr\left\{\left[ \Gamma^{H_{0}+tV,H_{0}+tV}(\psi_{f})\right](V)\cdot V\right\}dt,
\end{align*}
where $\psi_{f}(t_{1},t_{2})=f^{[2]}(t_{1},t_{2},t_{1})$, as expected.
\end{proof}
	
Based on the above lemma we have the following main result of this subsection.

\begin{thm}\label{SOT-Sacase}
Let $H$ and $H_{0}$ be two self-adjoint operators in $\hilh$ with $V:=H-H_{0}\in\Sp^{2}(\hilh)$. Let $f:\R\to\C$ be a twice differentiable function such that $f^{[2]}$ admits a factorization \eqref{facto2bis} for some separable Hilbert space $\mathfrak{H}$ and two bounded Borel functions $a,b:\R^2\to\mathfrak{H}$.
Assume either $H_{0}$ is bounded or $f'$ is bounded. Then, there exists a non-negative function $\eta:=\eta_{H_{0},V}\in L^{1}(\R)$ such that
\begin{align*}
\Tr\left\{\mathcal{R}_{2}^{\lin}(f,H_{0},V)\right\}=\int_{\R}f''(x)\eta(x)dx.
\end{align*}
\end{thm}

\begin{proof}	
According to Lemma \ref{Ko-Lem1}, we have
\begin{align*}
\Tr\left(\mathcal{R}_{2}^{\lin}(f,H_{0},V)\right)=2\int_0^1(1-t)\Tr\left\{\left[\Gamma^{H_{0}+tV,H_{0}+tV}(\psi_{f})\right](V)\cdot V\right\}dt,
\end{align*}
where $\psi_{f}(t_{1},t_{2})=f^{[2]}(t_{1},t_{2},t_{1})$.

Next, assume that $f''$ is continuous. According to (the proof of) \cite[Theorem 5.1]{CoLemSkSu19}, there exists a finite measure $\gamma$ on $\mathbb{R}$ and a function $\eta\in L^1(\mathbb{R})$ such that
\begin{align*}
\Tr\left\{\mathcal{R}_{2}^{\lin}(f,H_{0},V)\right\}
&=2\int_0^1(1-t)\,\Tr\left\{\left[\Gamma^{H_{0}+tV,H_{0}+tV}(\psi_f)\right](V)\cdot V\right\}\,dt \\
& = \int_{\R}f''(x) d\gamma(x)=\int_{\R} f''(x)\eta(x)dx.
\end{align*}
The measure $\gamma$ is defined by
$$\gamma=\alpha+\kappa dx,$$
where $\kappa\in L^1(\R)$ is non-negative, and where $\alpha$ is the measure defined for any Borel subset $\Omega\subset\R$ by
$$\alpha(\Omega)=\frac{1}{2} \nu((\Omega\times\Omega)\cap \Delta)$$
where $\Delta$ is the diagonal of $\R^2$ and $\nu$ is the measure on $\R^2$ defined for rectangles $C\times D$ by
\begin{align*}
\nu(C\times D)
&=2\int_0^1 (1-t) \Tr(\chi_C(H_0+tV)V\chi_D(H_0+tV)V) dt \\
&=2\int_0^1(1-t)\Tr((\mathcal{V}_{t,C,D})^{*}\mathcal{V}_{t,C,D})dt 
\end{align*}
where we defined \[\mathcal{V}_{t,C,D}:=\chi_D(H_0+tV)V\chi_C(H_0+tV)\in\Sp^2(\hilh).\] In particular, $\nu(C\times D)\ge 0$ and the measure $\nu$ is non-negative. We deduce that the measure $\gamma$ is non-negative, which in turn yields the non-negativity of the function $\eta$.
		
To conclude the proof, it suffices to show that for every twice differentiable function $g:\R\to\C$ with bounded $g''$ (without any factorization assumption on $g^{[2]}$), we have
\begin{align}\label{KoplienkoMOIversion}
2\int_0^1(1-t)\,\Tr\left\{\left[\Gamma^{H_0+tV,H_0+tV}(\psi_g)\right](V)\cdot V\right\}\,dt =\int_{\R}g''(x)\eta(x)dx.
\end{align}
Let $$g_n(x):=n\int_{0}^{x} \left(g(t+1/n)-g(t)\right) dt,\ x\in\R.$$
Notice that $\{g_n\}_{n\ge1}\subset C^2(\R)$ with $g_n''(x) = n(g'(x+1/n)-g'(x))$, so that $|g_{n}''(x)|\le\|g''\|_{\infty,{\R}}$ for all $x\in\R$. In particular, according to the first part of the proof,
\begin{align}\label{KoplienkoMOIversion2}
2\int_0^1(1-t)\, \Tr\left\{\left[\Gamma^{H_0+tV,H_0+tV}(\psi_{g_n})\right](V)\cdot V \right\}\,dt =\int_{\R}g_n''(x)\eta(x)dx.
\end{align}
We have $g_n''(x) \to g''(x)$ pointwise on $\R$. Similarly, $|\psi_{g_n}(t_{1},t_{2})|\le\frac{1}{2}\|g''\|_{\infty,\R}$, for all $(t_{1},t_{2})\in\R^{2}$. It is also straightforward to check that $\psi_{g_{n}}\to\psi_{g}$ pointwise on $\mathbb{R}^2$. Now, the assumption on $\{g_{n}\}_{n\ge1}$ alongwith \cite[Lemma 2.3]{Co22} give
\[\left[\Gamma^{H_{0}+tV,H_{0}+tV}(\psi_{g_{n}})\right](V)\underset{n\to+\infty}{\longrightarrow}\left[\Gamma^{H_{0}+tV,H_{0}+tV}(\psi_{g})\right](V)\]
weakly in $\Sp^{2}(\hilh)$. Therefore,
\[\Tr\left\{\left[\Gamma^{H_{0}+tV,H_{0}+tV}(\psi_{g_{n}})\right](V)\cdot V\right\}\underset{n\to+\infty}{\longrightarrow}\Tr\left\{\left[\Gamma^{H_{0}+tV,H_{0}+tV}(\psi_{g})\right](V)\cdot V\right\}.\]
Finally, using Lebesgue's dominated convergence theorem and by passing the limit in \eqref{KoplienkoMOIversion2}, we obtain formula \eqref{KoplienkoMOIversion}.
\end{proof}

The following corollary modifies Theorem \ref{SOT-Sacase} by replacing $\Tr\left(\mathcal{R}_{2}^{\lin}(f,H_{0},V)\right)$ with $$2\int_0^1(1-t)\, \Tr\left(\left[\Gamma^{H_0+tV,H_0+tV,H_0+tV}(f^{[2]})\right](V,V)\right)\,dt.$$ 
This substitution removes the assumptions on $H_0$ and $f'$. For example, if $f(x)=x^{2}$, the undefined difference \[f(H_{0}+V)-f(H_{0})-\frac{d}{ds}\bigg|_{s=0}f(H_{0}+sV)\]
is replaced by $V^{2}$.

\begin{crl}	
Let $H$ and $H_{0}$ be two self-adjoint operators in $\hilh$ such that $V:=H-H_{0}\in\Sp^{2}(\hilh)$. Let $f:\R\to\C$ be a twice differentiable function such that $f^{[2]}$ admits a factorization \eqref{facto2bis} (without any assumption on $H_0$ and $f'$). Then, there exists a non-negative function $\eta\in L^1(\R)$ such that
\begin{align*}
2\int_0^1(1-t)\, \Tr\left(\left[\Gamma^{H_0+tV,H_0+tV,H_0+tV}(f^{[2]})\right](V,V)\right)\,dt =\int_{\R}f''(x)\eta(x)dx.
\end{align*}
\end{crl}

\begin{proof}
When $f''$ is continuous, this is simply a consequence of  \cite[Theorem 5.1]{CoLemSkSu19} and Theorem \ref{SOT-Sacase}. In the general case, simply recall that
$$\Tr\left(\left[\Gamma^{H_0+tV,H_0+tV,H_0+tV}(f^{[2]})\right](V,V) \right)=\Tr\left\{ \left[\Gamma^{H_0+tV,H_0+tV}(\psi_f)\right](V)\cdot V\right\},$$
where $\psi_f(t_{1},t_{2})=f^{[2]}(t_{1},t_{2},t_{1})$ and then approximate $f$ by a sequence of $C^{2}(\R)$ functions as in the proof of Theorem \ref{SOT-Sacase} (this also implies that the function we integrate is indeed integrable).
\end{proof}

\begin{rmrk}
From the proof of Theorem \ref{SOT-Sacase}, we also see that for every $f : \mathbb{R} \to \mathbb{C}$ twice differentiable with bounded $f''$ (without any assumption on $H_0$, $f'$ or $f^{[2]}$), we have
\begin{align*}
2\int_0^1(1-t)\, \Tr\left\{\left[\Gamma^{H_0+tV,H_0+tV}(\psi_f)\right](V)\cdot V\right\}\,dt =\int_{\R}f''(x)\eta(x)dx.
\end{align*}
\end{rmrk}

\subsection{Neidhardt trace formula}
		
We now establish the existence of second order spectral shift functions for a pair of unitaries $\{e^{iA}U_{0},U_{0}\}$, where $A\in\Sp^{2}_{sa}(\hilh)$ and $U_{0}\in\unih$. The existence of such SSFs was first proved by Neidhardt \cite{Ne88}, for functions $f:\cir\to\C$ such that $f''$ has an absolutely convergent Fourier series. 

This result was later extended to functions in the Besov space $B^{2}_{\infty 1}(\cir)$ (see \cite{Pe05}).  More recently, in \cite{ChPrSk24}, it was further generalized to twice differentiable functions $f:\cir\to\C$ with bounded $f''$, under the condition that $f^{[2]}\in \mathcal{B}(\cir) \widehat{\otimes}_{i}\mathcal{B}(\cir)\widehat{\otimes}_{i} \mathcal{B}(\cir)$ (denoted as $\mathfrak{F}_{2}(\cir)$ in \cite{ChPrSk24}). In this section, we further broaden the class of admissible functions to include all twice differentiable functions $f:\cir\to\C$ for which $f^{[2]}$ admits a Hilbert factorization \eqref{facto2bis}.

\medskip
				
The following lemma follows from Theorem \ref{caracS1} and the same arguments than that of the proof of \cite[Proposition 2.5]{CoLemSkSu19}. We leave the details to the reader.

\begin{lma}\label{beltoS1andtrace}
Let $A,B\in\mathcal{U}(\hilh)$. Let $\varphi:\cir^{2}\to\C$ and $\phi:\cir^{3}\to\C$ be bounded Borel functions satisfying respectively \eqref{facto1} and \eqref{facto1bil}. Then, for every $X\in\Sp^{1}(\hilh)$ and $Y,Z\in\Sp^{2}(\hilh)$,
$$\left[\Gamma^{A,A}(\varphi)\right](X)\in\mathcal{S}^{1}(\hilh)\quad \text{and}\quad \left[\Gamma^{A,B,A}(\phi)\right](Y,Z) \in\Sp^{1}(\hilh)$$
and we have
$$\Tr(\left[\Gamma^{A,A}(\varphi)\right](X))=\Tr(\widetilde{\varphi}(A)\cdot X)\quad\text{and}\quad \Tr(\left[\Gamma^{A,B,A}(\phi)\right](Y,Z))=\Tr\left\{\left[\Gamma^{A,B}(\widetilde{\phi})\right](Y)\cdot Z\right\},$$
where $\widetilde{\varphi}(t)=\varphi(t,t)$ and $\widetilde{\phi}(t_{1},t_{2})=\phi(t_{1},t_{2},t_{1})$.
\end{lma}
	
We now prove the main result.

\begin{thm}\label{N-Traceformula}
Let $U_{0}\in\mathcal{U}(\hilh)$, $A\in\Sp^{2}_{sa}(\hilh)$ and consider $U(s)=e^{isA}U_{0}$ for $s\in\R$. Denote $U=U(1)$. Let $f:\cir\to\C$ be a twice differentiable function on $\cir$ such that $f^{[2]}$ admits a factorization \eqref{facto2bis} for some separable Hilbert space $\mathfrak{H}$ and two bounded Borel functions $a,b:\cir^2\to\mathfrak{H}$.
Then
$$\mathcal{R}_{2}^{\mult}(f,U_{0},A)\in\Sp^{1}(\hilh)$$
and there exist $\xi_{1,U_{0},A},\xi_{2,U_{0},A}\in L^{1}(\cir)$ with $\xi_{1,U_{0},A}\in \text{\normalfont span}\{\overline{z}\}$ such that 
\begin{align}
\label{N-Result}\Tr\left\{\mathcal{R}_{2}^{\mult}(f,U_{0},A)\right\}=\sum_{k=1}^{2}\int_{\cir}f^{(k)}(z)\xi_{k,U_{0},A}(z)dz,
\end{align}
where the derivative $\frac{d}{ds}\big|_{s=0}f(U(s))$ is taken in the $\Sp^{2}(\hilh)$.
\end{thm}	
	
\begin{proof}
The assumption on $f^{[2]}$ confirms that $f''$ is bounded on $\cir$. Therefore, according to \cite[Theorem 5.1]{Co24}, the function $\varphi:s\in\R\mapsto f(U(s))-f(U_{0})\in\Sp^{2}(\hilh)$ is twice differentiable with
\begin{equation*}
\varphi''(s)=-\left[\Gamma^{U(s),U(s)}(f^{[1]})\right](A^2U(s))-2\left[\Gamma^{U(s),U(s),U(s)}(f^{[2]})\right](AU(s),AU(s))
\end{equation*}
and, we further have
\begin{align}
\nonumber&\mathcal{R}_{2}^{\mult}(f,U_{0},A)\\
\label{N-R1}&=\left[\Gamma^{U,U_0}(f^{[1]})\right](U-U_{0}-iAU_0)-\left[\Gamma^{U,U_0,U_0}(f^{[2]})\right](U-U_{0},iAU_{0}).
\end{align}
It follows from the definition of divided differences that for every $t_{1},t_{2}\in\cir$,
\begin{align}
\nonumber f^{[1]}(t_1,t_2)&=(t_1-1)f^{[2]}(t_1,1,t_2)+f^{[1]}(1,t_2)\\
\nonumber&=\langle(t_1-1)a(t_{1},1),b(1,t_{2})\rangle+f^{[1]}(1,t_2)\\
\label{fac2-fac1}&=\langle\widetilde{a}(t_1),\widetilde{b}(t_{2})\rangle+g(t_2),
\end{align}
where $\widetilde{a}(t_1)=(t_1-1)a(t_{1},1), \widetilde{b}(t_{2})=b(1,t_{2})$ and $g(t_2)=f^{[1]}(1,t_2)$. In particular, since $U-U_{0}-iAU_0=(e^{iA}-I-iA)U_{0} \in\Sp^1(\hilh)$ and $U-U_{0}\in \Sp^2(\hilh)$, it follows from \eqref{N-R1}, Theorem \ref{caracS1} and Remark \ref{caracS1rmrk} that
$$\mathcal{R}_{2}^{\mult}(f,U_{0},A)\in\Sp^{1}(\hilh).$$
Let $\xi_{1,U_{0},A},\xi_{2,U_{0},A}\in L^{1}(\cir)$ be given by \cite[Theorem 3.2]{ChPrSk24}. In particular, for every trigonometric polynomial $P$, we have
\begin{equation}\label{N-R2}
\Tr(\mathcal{R}^{\mult}_{2}(P,U_{0},A))=\sum_{k=1}^{2}\int_{\cir}P^{(k)}(z)\xi_{k,U_{0},A}(z)dz.
\end{equation}
Now it is very much obvious to guess that (also we shall see in the proof), the proof now solely relies on the approximation of $f$ (satisfying the hypothesis of the theorem) by a polynomial. However, from \eqref{N-R1} we see that while we can approximate $$\Tr\left(\left[\Gamma^{U,U_0,U_0}(f^{[2]})\right](U-U_{0},iAU_{0})\right)$$ 
but we can not directly approximate the first term $$\Tr\left(\left[\Gamma^{U,U_0}(f^{[1]})\right](U-U_{0}-iAU_0\right).$$
The following construction addresses this issue. Define
\begin{align}
\label{N-R3}\mathcal{R}_{l2}(f,U_{0},(U-U_{0})):=\left[\Gamma^{U_{0},U}(f^{[1]})\right](U-U_{0})-\left[\Gamma^{U_{0},U_{0}}(f^{[1]})\right](U-U_{0}).
\end{align}
The assumptions on $f$ and $U-U_{0}\in\Sp^{2}(\hilh)$ together ensure that \eqref{N-R3} is an element of $\Sp^{2}(\hilh)$. Again, from \cite[Proposition 3.5]{Co24}, we have
\begin{align}
\label{N-R4}\mathcal{R}_{l2}(f,U_{0},(U-U_{0}))=\left[\Gamma^{U_{0},U,U_{0}}(f^{[2]})\right](U-U_{0},U-U_{0}),
\end{align}
which is a trace class operator, by Theorem \ref{caracS1}.
Next, we got
\begin{align}
\nonumber{\mathcal{R}_{2}^{\mult}(f,U_{0},A)}&=\left[\Gamma^{U_{0},U}(f^{[1]})\right](U-U_{0})-\left[\Gamma^{U_{0},U_{0}}(f^{[1]})\right](iAU_{0})\\
\label{N-R5}&=\mathcal{R}_{l2}(f,U_{0},(U-U_{0}))+\left[\Gamma^{U_{0},U_{0}}(f^{[1]})\right](U-U_{0}-iAU_{0}),
\end{align}
where both the terms
\[\mathcal{R}_{l2}(f,U_{0},(U-U_{0}))\quad\text{and}\quad\left[\Gamma^{U_{0},U_{0}}(f^{[1]})\right](U-U_{0}-iAU_{0})\]
belong to $\Sp^{1}(\hilh)$, by \eqref{N-R4}, \eqref{fac2-fac1}, and Theorem \ref{caracS1}. Furthermore, applying Lemma \ref{beltoS1andtrace} we got
\[\Tr\left(\mathcal{R}_{l2}(f,U_{0},(U-U_{0}))\right)=\Tr\left\{\left[\Gamma^{U_{0},U}(h)\right](U-U_{0})\cdot(U-U_{0})\right\},\]
and
\[\Tr\left(\left[\Gamma^{U_{0},U_{0}}(f^{[1]})\right](U-U_{0}-iAU_{0})\right)=\Tr\left\{f'(U_{0})\cdot(U-U_{0}-iAU_{0})\right\},\]
where $h(t_{1},t_{2}):=f^{[2]}(t_{1},t_{2},t_{1})$.
We now use the representation \eqref{N-R5} for $\mathcal{R}_{2}^{\mult}(f,U_{0},A)$ instead of \eqref{N-R1}.

The assumption on $f$ ensures that there exists a sequence of functions $\{f_{j}\}_{j\ge1}\subset C^{2}(\cir)$ such that
$$f_{j}'\underset{j\to+\infty}{\longrightarrow}f'\ \text{uniformly on}\ \cir, \quad \quad f_{j}''\underset{j\to+\infty}{\longrightarrow}f''\ \text{pointwise on}\ \cir,$$
and
$$f_{j}^{[2]}\underset{j\to+\infty}{\longrightarrow}f^{[2]} \ \text{pointwise on}\ \cir^3.$$
Moreover, there exists a constant $M>0$ such that for every $j\in\N$ 
$$\left\|f_{j}^{(k)}\right\|_{\infty,\cir}\le M\quad\text{for}~k=1,2.$$
Such a sequence exists, see \cite[Lemma 2.2]{Co24}. Note that the sequence $\{f_{j}^{[2]}\}_{j\ge1}$ does not necessarily have the factorization \eqref{facto2bis}. However, for every $j$ the function $f_{j}\in C^{2}(\cir)$ can be approximated by a trigonometric polynomial $f_{j,N}:=f_{j}\ast K_{N}$, where $K_{N}$ is the Fej\'er kernel such that
$f_{j,N}^{(k)}$ converges uniformly to $f_{j}^{(k)}$ on $\cir$, for $k=1,2$. Additionally, we have for $j,N\in\N$
\[\|f_{j,N}^{(k)}\|_{\infty,\cir}\le\|f_{j}^{(k)}\|_{\infty,\cir}\|K_{N}\|_{1}\le\|f_{j}^{(k)}\|_{\infty,\cir}\le M,\quad\text{for}~k=1,2.\]
As a result we always have the existence of a sequence of polynomials $\{P_{N}\}_{N\ge1}$ such that
$$P_{N}'\underset{N\to+\infty}{\longrightarrow}f'\ \text{uniformly on}\ \cir, \quad \quad P_{N}''\underset{N\to+\infty}{\longrightarrow}f''\ \text{pointwise on}\ \cir,$$
$$P_{N}^{[2]}\underset{N\to+\infty}{\longrightarrow}f^{[2]}\ \text{pointwise on}\ \cir^3,$$
and there exists $M>0$ such that $\|P_{N}^{(k)}\|_{\infty,\cir}\le M$, for $k=1,2$. Additionally, $\{P_{N}^{[2]}\}_{N\ge1}$ has the factorization \eqref{facto2bis}. Applying \eqref{N-R2}, \eqref{N-R4}, \eqref{N-R5} and Lemma \ref{beltoS1andtrace} to each $P_{N}$ yields
\begin{align}
\nonumber&\Tr\left\{\left[\Gamma^{U_{0},U}(h_{N})\right](U-U_{0})\cdot(U-U_{0})\right\}+\Tr\left\{P_{N}'(U_{0})\cdot(U-U_{0}-iAU_{0})\right\}\\
\label{N-R6}&\hspace{0,5cm}=\sum_{k=1}^{2}\int_{\cir}P_{N}^{(k)}(z)\xi_{k,U_{0},A}(z)dz,
\end{align}
where $h_{N}(t_1,t_2):=P_{N}^{[2]}(t_1,t_2,t_1)$.
Using Lebesgue's dominated convergence theorem, we have
\begin{align}
\label{N-R7}\sum_{k=1}^{2}\int_{\cir}P_{N}^{(k)}(z)\xi_{k,U_{0},A}(z)dz\underset{N\to+\infty}{\longrightarrow}\sum_{k=1}^{2}\int_{\cir}f^{(k)}(z)\xi_{k,U_{0},A}(z)dz.
\end{align}
Next, the assumption on $\{P_{N}\}_{N\ge1}$ together with \cite[Lemma 2.2]{ChCoGiPr24} give
\begin{align*}
\left[\Gamma^{U_{0},U}(h_{N})\right](U-U_{0})\underset{N\to+\infty}{\longrightarrow}\left[\Gamma^{U_{0},U}(h)\right](U-U_{0})
\end{align*}
weakly in $\Sp^{2}(\hilh)$. As a result,
\begin{align}
\label{N-R8}\Tr\left\{\left[\Gamma^{U_{0},U}(h_{N})\right](U-U_{0})\cdot(U-U_{0})\right\}\underset{N\to+\infty}{\longrightarrow}\Tr\left\{\left[\Gamma^{U_{0},U}(h)\right](U-U_{0})\cdot(U-U_{0})\right\}.
\end{align}
Finally, the uniform convergence of $\{P_{N}'\}_{N\ge1}$ to $f'$ implies that
\begin{align}
\label{N-R9}\Tr\left\{P_{N}'(U_{0})\cdot(U-U_{0}-iAU_{0})\right\}\underset{N\to+\infty}{\longrightarrow}\Tr\left\{f'(U_{0})\cdot(U-U_{0}-iAU_{0})\right\}.
\end{align}
Hence using \eqref{N-R7}, \eqref{N-R8}, \eqref{N-R9} and passing the limit in \eqref{N-R6}, we get the desired formula. This completes the proof.
\end{proof}

\begin{rmrk}
Notice that, the trace formula \eqref{N-Result} could be rewritten in the following nicer form
\[\Tr\left\{\mathcal{R}_{2}^{\mult}(f,U_{0},A)\right\}= c\widehat{f}(1) + \int_{\cir}f^{''}(z)\xi_{2,U_{0},A}(z)dz,\]
for some constant $c$ which can be computed explicitly from the proof of \cite[Theorem 3.2]{ChPrSk24}. Indeed,
\[c=-\int_0^1 (1-t)\Tr(A^2 e^{itA}U_0) dt,\]
which, again simplifies to $\Tr\left((e^{iA}-I-iA)U_{0}\right).$
\end{rmrk}

\subsection{Koplienko-Neidhardt trace formula: contraction case}

Finally, in this subsection, we derive the trace formulas for a pair of contractions in both linear and multiplicative paths. This section includes the following results:

We start by establishing the second order SSF for a pair of contractions $\{T_{1},T_{0}\}$, where $T_{1}-T_{0}\in\Sp^{2}(\hilh)$, via linear path. Our result applies to functions $f\in\A(\D)$ for which $f',f''\in\A(\D)$, satisfying $f^{[2]}\in\A(\D)\widehat{\otimes}\A(\D)\widehat{{\otimes}}\A(\D)$, the projective tensor product of three copies of $\A(\D)$, see Theorem \ref{SOT-Concase3}. It is of general interest to know how much one can relax the condition on operators and functions. Taking a step further we explore how much we can enlarge the class of admissible functions. Surprisingly, with only mild restrictions on the contractions, we are able to derive trace formulas for a much wider class of functions, see Theorems \ref{SOT-Concase1} and \ref{SOT-Concase2}. 

On the other hand, for the multiplicative path, we establish the result for twice differentiable functions $f:\cir\to\C$ such that $f^{[2]}$ admits the factorization \eqref{facto2bis}, see Theorem \ref{SOT-Concase-M}.

\begin{thm}\label{SOT-Concase3}
Let $T_{1},T_{0}\in\conth$ such that $V:=T_{1}-T_{0}\in\Sp^{2}(\hilh)$. If $f\in\A(\D)$ satisfies that $f',f''\in\A(\D)$ and $f^{[2]}\in\A(\D)\widehat{\otimes}\A(\D)\widehat{\otimes}\A(\D)$, then there exists a function $\eta:=\eta_{T_{0},V}\in L^1(\cir)$ such that
\begin{align*}
\Tr\left\{\mathcal{R}_{2}^{\lin}(f,T_{0},V)\right\}=\int_{\cir}f''(z)\eta(z)dz.
\end{align*}
\end{thm}

\begin{proof}
Denote $T_{t}:=T_{0}+tV$ for $t\in[0,1]$. The assumption on $f$ and Lemma \ref{Lem-Div-Discalg} implies that $f^{[2]}\in\A(\D^{3})$. It is important to observe that $f^{[2]}\in\A(\D)\widehat{\otimes}\A(\D)\widehat{\otimes}\A(\D)$ implies
\begin{align*}
f^{[2]}(\lambda_{1},\lambda_{2},\lambda_{3})=\sum_{n=0}^{+\infty}f_{n}(\lambda_{1})g_{n}(\lambda_{2})h_{n}(\lambda_{3}),
\end{align*}	
where $\{f_{n}\},\{g_{n}\},\{h_{n}\}$ are in $\A(\D)$, and
\begin{align*}
\sum_{n=0}^{+\infty}\left\|f_{n}\right\|_{\infty,\overline{\D}}\left\|g_{n}\right\|_{\infty,\overline{\D}}\left\|h_{n}\right\|_{\infty,\overline{\D}}<\infty.
\end{align*}	
Then, according to \cite[Proposition 7.4]{CoLemSu21} for $U_{1},U_{2},U_{3}\in\mathcal{U}(\hilh)$ and $V_{1},V_{2}\in\Sp^{2}(\hilh)$ we establish
\begin{align}
\label{SOT-Concase3-R3}\left[\Gamma^{U_{1},U_{2},U_{3}}(f^{[2]})\right](V_{1},V_{2})=\sum_{n=0}^{+\infty}f_{n}(U_{1})V_{1}g_{n}(U_{2})V_{2}h_{n}(U_{3}).
\end{align}
According to Proposition \ref{Perturbation-Formula}, we have
\[f(T_{1})-f(T_{0})=\left[\Gamma^{T_{1},T_{0}}(f^{[1]})\right](V)=\left[\Gamma^{T_{0},T_{1}}(f^{[1]})\right](V),\]
which leads to
\begin{align}
\label{SOT-Concase3-R4}\mathcal{R}^{\lin}_{2}(f,T_{0},V)=\left[\Gamma^{T_{0},T_{1},T_{0}}(f^{[2]})\right](V,V).
\end{align}
Let $U_{0}$ and $U_{1}$ be the Sch\"{a}ffer matrix unitary dilations of $T_{0}$ and $T_{1}$, respectively, on the Hilbert space $\hilk=\ell_{2}(\hilh)\oplus\hilh\oplus\ell_{2}(\hilh)$. Then, using \eqref{MOI-Cont-Def2} and \eqref{SOT-Concase3-R4}, we obtain
\begin{align}
\nonumber\Tr\left\{\mathcal{R}^{\lin}_{2}(f,T_{0},V)\right\}&=\Tr\left\{P_{\hilh}\left[\Gamma^{U_{0},U_{1},U_{0}}(f^{[2]})\right](\widetilde{V},\widetilde{V})\big|_{\hilh}\right\}\qquad\left[\widetilde{V}=P_{\hilh}VP_{\hilh}\right]\\
\nonumber&\overset{\eqref{SOT-Concase3-R3}}{=}\Tr\left\{P_{\hilh}\sum_{n=0}^{+\infty}f_{n}(U_{0})\widetilde{V}g_{n}(U_{1})\widetilde{V}h_{n}(U_{0})\big|_{\hilh}\right\}\\
\nonumber&=\Tr\left\{\sum_{n=0}^{+\infty}f_{n}(T_{0})Vg_{n}(T_{1})Vh_{n}(T_{0})\right\}\\
\nonumber&=\Tr\left\{\sum_{n=0}^{+\infty}f_{n}h_{n}(T_{0})Vg_{n}(T_{1})V\right\}\qquad\left[\text{using cyclicity of trace}\right]\\
\nonumber&=\Tr\left\{\left(P_{\hilh}\sum_{n=0}^{+\infty}f_{n}h_{n}(U_{0})\widetilde{V}g_{n}(U_{1})\big|_{\hilh}\right)\cdot V\right\}\\
\label{SOT-Concase3-R5}&=\Tr\left\{\left[\Gamma^{T_{0},T_{1}}(\psi_{f})\right](V)\cdot V\right\},	
\end{align}
where we set $\psi_{f}(\lambda_{1},\lambda_{2})=f^{[2]}(\lambda_{1},\lambda_{2},\lambda_{1})$.
Consider the sequence of polynomials $\{\varphi_{k}\}_{k\ge1}\subset\A(\D)$ (see \eqref{Cesaro-Seqn}) such that
\begin{align*}
\left\|\varphi_k''-f''\right\|_{\infty,\cir}\underset{k\to+\infty}{\longrightarrow}0.
\end{align*}
Utilizing Theorem 1 from \cite{PoSu12} to each $\varphi_{k}$, we obtain $\eta\in L^{1}(\cir)$ such that
\begin{align}
\label{SOT-Concase3-R6}\Tr\left\{\mathcal{R}^{\lin}_{2}(\varphi_{k},T_{0},V)\right\}=\int_{\cir}\varphi_{k}''(z)\eta(z)dz.
\end{align}
The general case follows from the identities \eqref{SOT-Concase3-R5} and \eqref{SOT-Concase3-R6}, along with the final approximation result. This completes the proof.
\end{proof}

In \cite[Theorem 1]{PoSu12}, the aforementioned result was established for polynomials. Our result significantly enlarges the admissible function class. Note that, functions in the analytic Besov class satisfy the hypotheses of Theorem \ref{SOT-Concase3}, see \cite[Eq. 2.10]{Pe09}.

\begin{rmrk}\label{rem:high_cont}
Let $n\geq 2$. For $T_1 - T_0 \in \Sp^n(\hilh)$, and $f \in \A(\D)$ satisfying $f',\ldots, f^{(n)}\in\A(\D)$,  $f^{[n]} \in \A(\D) \widehat{\otimes} \cdots \widehat{\otimes} \A(\D)$, the trace formula for the higher order Taylor remainder 
\begin{align*}
\Tr \left\{ \mathcal{R}_{n}^{\lin}(f, T_{0}, V) \right\} = \int_{\cir} f^{(n)}(z) \eta(z) \, dz.
\end{align*}
holds true. The proof is analogous to the case $n=2$: starting with complex polynomials, as established in \cite[Theorem 1.3]{PoSkSu14}. Extending to the required function class then follows via similar limiting arguments, using \eqref{Taylor-Exp-Cont} and Theorem \ref{MOIEst-Cont}. Thus, our result substantially broaden the admissible function class considered in \cite{PoSkSu14}.
\end{rmrk}

\subsubsection{Koplienko-Neidhardt trace formula for a pair of contractions with one of being normal}\hfil

The main theorem of this section is the following.

\begin{thm}\label{SOT-Concase1}
Let $T_{1},T_{0}\in\conth$ be such that $T_{0}$ is normal and  $V:=T_{1}-T_{0}\in\Sp^{2}(\hilh)$. Let $f\in\A(\D)$ be such that $f',f''\in\A(\D)$, and suppose that $f_{\cir}^{[2]}$ admits a factorization \eqref{facto2bis}. Then, there exists a function $\eta:=\eta_{T_{0},V}\in L^1(\cir)$ such that
\begin{align*}
\Tr\left\{\mathcal{R}_{2}^{\lin}(f,T_{0},V)\right\}=\int_{\cir}f''(z)\eta(z)dz.
\end{align*}
\end{thm}

Due to the lack of strength in the functional calculus for contractions, we are unable to use the function-theoretic approach employed in the unitary case to conclude the main result. Therefore, we consider an alternative approach. Our approach is largely based on Voiculescu's finite-dimensional approximation technique \cite{Vo79}, though with some modifications. We shall need the following theorem for our use.

\begin{thm}\label{SOT-Concase-FDH}
Let $\{T_{1},T_0\}$ be a pair of contractions on a finite dimensional Hilbert space $\hilh$. Denote $T_{s}=T_{0}+sV$, where $s\in[0,1]$ and $V=T_{1}-T_{0}$. Let $f\in\A(\D)$ be such that $f',f''\in\A(\D)$. Then, there exists a function $\eta:=\eta_{T_{0},V}\in L^1(\cir)$ such that 
\begin{align}\label{SOT-Concase-R1}
\Tr\left\{\mathcal{R}_{2}^{\lin}(f,T_{0},V)\right\}=\int_{\cir}f''(z)\eta(z)dz,
\end{align}
and 
\begin{equation*}
\eta(z)=\int_{0}^{1}\operatorname{Tr} \Big[V\Big\{\mathcal{E}_0(Arg(z))-\mathcal{E}_s(Arg(z))\Big\}\Big]ds, ~z\in\cir,
\end{equation*}
where $\mathcal{E}_s(\cdot)$ denotes the semi-spectral measure associated with the contraction $T_{s}$ and $Arg(z)$ is the principal argument of $z$.	
\end{thm}

\begin{proof}
Let $\{\varphi_{k}\}_{k\ge1}$ be the sequence of polynomials considered in Proposition \ref{Perturbation-Formula} (see Eq. \eqref{Cesaro-Seqn}). Then, 
\begin{align}
\label{SOT-Concase-R2}\left\|(\varphi_k)^{[m]}-f^{[m]}\right\|_{\infty,\overline{\D}^{m+1}}\le c_{m}\left\|(\varphi_k)^{(m)}-f^{(m)}\right\|_{\infty,\cir}\underset{k\to+\infty}{\longrightarrow}0\quad\text{for}~m=1,2.
\end{align} 
Note that, according to \cite[Theorem 3.3]{ChDaPr24}, \eqref{SOT-Concase-R1} holds for every complex polynomial. On the other hand we have
\begin{align*}
\left\|\mathcal{R}_{2}^{\lin}(f,T_{0},V)-\mathcal{R}_{2}^{\lin}(\varphi_{k},T_{0},V)\right\|_{2}\underset{k\to+\infty}{\longrightarrow}0,
\end{align*}
which follows from \eqref{Taylor-Exp-Cont}, \eqref{SOT-Concase-R2} and Theorem \ref{MOIEst-Cont}. The last equation further implies
\begin{align}
\label{SOT-Concase-R3}\left\|\mathcal{R}_{2}^{\lin}(f,T_{0},V)-\mathcal{R}_{2}^{\lin}(\varphi_{k},T_{0},V)\right\|\underset{k\to+\infty}{\longrightarrow}0.
\end{align}
We should emphasize that $\Tr(\cdot)$ is a finite sum here. Thus, \eqref{SOT-Concase-R3} simplifies to
\begin{align*}
\Tr\left\{\mathcal{R}_{2}^{\lin}(\varphi_{k},T_{0},V)\right\}\underset{k\to+\infty}{\longrightarrow}\Tr\left\{\mathcal{R}_{2}^{\lin}(f,T_{0},V)\right\}.
\end{align*}
The approximation in the right hand side of \eqref{SOT-Concase-R1} is trivial. This completes the proof.
\end{proof}

We shall need the following lemma from \cite[Lemma 4.3]{ChDaPr24}.

\begin{lma}\label{appthm}
Let $T_{0}$ be a normal contraction on a separable infinite dimensional Hilbert space $\hilh$, and let $V\in\Sp^{2}(\hilh)$. Let $T_{1}=T_{0}+V$. Then, there exists a sequence $\{P_{n}\}_{n\ge1}$ of finite rank projections such that $\{P_{n}\}\uparrow I$, and each of the following terms
\begin{align*}
&\text{\normalfont(i)}~~\|P_{n}^\perp T_{0}P_{n}\|_{2},~~\text{\normalfont(ii)}~~\|P_{n}^\perp V\|_{2},~~\text{\normalfont(iii)}~~\|P_{n}^\perp V^{*}\|_{2}~~\text{\normalfont(iv)}~~\|P_{n}^\perp T_{1}P_{n}\|_{2}
\end{align*}
converge to zero as $n\rightarrow +\infty$.
\end{lma}

To proceed further we need a perturbation formula which is rather simple but very crucial.

\begin{lma}\label{Lem-Prjout-Cont}
Let $T_{1},T_{0}\in\conth$ with $V:=T_{1}-T_{0}\in\Sp^{2}(\hilh)$. Let $\{P_{n}\}_{n\ge1}$ be the sequence of finite rank projections considered in Lemma \ref{appthm}. Let $f\in\A(\D)$ be a polynomial. Then 
\begin{align}
\nonumber&\text{\normalfont(i)}\quad\left[\Gamma^{T_{1},T_{0}}(f^{[1]})\right](P_{n}V)-P_{n}\left[\Gamma^{T_{1,n},T_{0}}(f^{[1]})\right](P_{n}V)\\
\label{Lem-Prjout-Cont-Eq1}&\hspace*{1in}=\left[\Gamma^{T_{1},T_{1,n},T_{0}}(f^{[2]})\right]((T_{1}-T_{1,n})P_{n},P_{n}V),\\
\nonumber&\text{\normalfont(ii)}\quad\left[\Gamma^{T_{1,n},T_{0}}(f^{[1]})\right](P_{n}V)-P_{n}\left[\Gamma^{T_{1,n},T_{0,n}}(f^{[1]})\right](P_{n}V)\\
\label{Lem-Prjout-Cont-Eq2}&\hspace*{1in}=\left[\Gamma^{T_{1,n},T_{0},T_{0,n}}(f^{[2]})\right](P_{n}V,(T_{0}-T_{0,n})P_{n}),
\end{align}
where $T_{1,n}=P_{n}T_{1}P_{n}$ and $T_{0,n}=P_{n}T_{0}P_{n}$. Additionally, \eqref{Lem-Prjout-Cont-Eq1} and \eqref{Lem-Prjout-Cont-Eq2} remain valid if $T_{1}$ is replaced with $T_{0}$.
\end{lma}

\begin{proof}
We only prove \eqref{Lem-Prjout-Cont-Eq1} and it is sufficient to prove it for $f(\lambda)=\lambda^{r}$, $r\in\N\cup\{0\}$. Notice that for $r=0$ both sides of \eqref{Lem-Prjout-Cont-Eq1} are identically zero. For $r\in \N$, a simple application of the divided differences of $f$ implies that
\begin{align}
\nonumber&\left[\Gamma^{T_{1},T_{0}}(f^{[1]})\right](P_{n}V)-P_{n}\left[\Gamma^{T_{1,n},T_{0}}(f^{[1]})\right](P_{n}V)\\
\nonumber&=\sum_{\substack{p_{0},p_{1}\ge0\\p_{0}+p_{1}=r-1}}T_{1}^{p_{0}}\left(P_{n}V\right)T_{0}^{p_{1}}-P_{n}\sum_{\substack{p_{0},p_{1}\ge0\\p_{0}+p_{1}=r-1}}T_{1,n}^{p_{0}}(P_{n}V)T_{0}^{p_{1}}\\
\nonumber&=\sum_{\substack{p_{0},p_{1}\ge0\\p_{0}+p_{1}=r-1}}\left[T_{1}^{p_{0}}-T_{1,n}^{p_{0}}\right]P_{n}VT_{0}^{p_{1}}\\
\nonumber&=\sum_{\substack{p_{0},p_{1}\ge0\\p_{0}+p_{1}=r-1}}\bigg[\sum_{\substack{q_{0},q_{1}\ge0\\q_{0}+q_{1}=p_{0}-1}}T_{1}^{q_{0}}\left(T_{1}-T_{1,n}\right)T_{1,n}^{q_{1}}\bigg]P_{n}VT_{0}^{p_{1}}\\
\nonumber&=\left[\Gamma^{T_{1},T_{1,n},T_{0}}(f^{[2]})\right]\left(\left(T_{1}-T_{1,n}\right)P_{n},P_{n}V\right).
\end{align}
This completes the proof.
\end{proof}

\begin{thm}\label{Approx-Thm-Concase}
Let $T_{1},T_{0}\in\conth$ be such that $T_{0}$ is normal and  $V:=T_{1}-T_{0}\in\Sp^{2}(\hilh)$. Denote $T_{s}=T_{0}+sV,~s\in[0,1]$. Let $f\in\A(\D)$ be such that $f',f''\in\A(\D)$, and suppose there exists a separable Hilbert space $\mathfrak{H}$ and two bounded Borel functions $a,b:\cir^{2}\to\mathfrak{H}$ such that $f_{\cir}^{[2]}$ admits a factorization \eqref{facto2bis}. Then, \[\mathcal{R}_{2}^{\lin}(f,T_{0},V)\in\Sp^{1}(\hilh).\] 
Furthermore, there exists a sequence $\{P_{n}\}_{n\ge1}$ of finite rank projections such that
\begin{align*}
\left\|\mathcal{R}_{2}^{\lin}(f,T_{0},V)-P_{n}\left\{f(T_{1,n})-f(T_{0,n})-\frac{d}{ds}\bigg|_{s=0}f(T_{s,n})\right\}\right\|_{1}\underset{n\to+\infty}{\longrightarrow}0,
\end{align*}
where $V_{n}=P_{n}VP_{n}$ and $T_{s,n}=P_{n}T_{s}P_{n}$, $s\in[0,1]$.
\end{thm}

\begin{proof}
The proof is known for all complex polynomials on $\cir$ (see, e.g., \cite[Theorem 4.5]{ChDaPr24}), but it does not apply in this general case. This proof primarily revolves around algebraic computations of MOIs along with Lemmas \ref{appthm} and \ref{Lem-Prjout-Cont}. Going further, we note the following crucial estimate:
\begin{equation}\label{Approx-Thm-Concase-R1}
\left\|\left[\Gamma^{T_{1},T_{2},T_{3}}(f^{[2]})\right](V_{1},V_{2})\right\|_{1}\le\|a\|_{\infty}\|b\|_{\infty}\|V_{1}\|_{2}\|V_{2}\|_{2},
\end{equation}
for $T_{1},T_{2},T_{3}\in\text{cont}(\hilh)$ and $V_{1},V_{2}\in\Sp^{2}(\hilh)$, which is a collective application of the Hilbert-factorization \eqref{facto2bis} of $f_{\cir}^{[2]}$, \eqref{MOI-Cont-Def2} and \cite[Theorem 6.2]{CoLemSu21}.

The proof is now divided into two step-by-step approximations of both operators and functions. We already have observed that there exists a sequence of polynomial $\{\varphi_{k}\}_{k\ge1}\subset\A(\D)$ such that
\begin{align*}
\left\|(\varphi_k)^{[m]}-f^{[m]}\right\|_{\infty,\overline{\D}^{m+1}}\le c_{m}\left\|(\varphi_k)^{(m)}-f^{(m)}\right\|_{\infty,\cir}\underset{k\to+\infty}{\longrightarrow}0\quad\text{for}~m=1,2.
\end{align*} 
We start with the following identity from Lemma \ref{Lem-Prjout-Cont} for $\{\varphi_{k}\}_{k\ge1}$:
\begin{align}\label{Approx-Thm-Concase-R2}
\left[\Gamma^{T_{1},T_{0}}(\varphi_{k}^{[1]})\right](P_{n}V)-P_{n}\left[\Gamma^{T_{1,n},T_{0}}(\varphi_{k}^{[1]})\right](P_{n}V)=\left[\Gamma^{T_{1},T_{1,n},T_{0}}(\varphi_{k}^{[2]})\right]((T_{1}-T_{1,n})P_{n},P_{n}V).
\end{align}
Due to Theorem \ref{MOIEst-Cont} we have
\begin{align}
\label{Approx-Thm-Concase-R3}\text{\normalfont L.H.S. of \eqref{Approx-Thm-Concase-R2}}\mathrel{\mathop{\longrightarrow}^{\mathrm{\|\cdot\|_{2}}}_{\mathrm{k\to+\infty}}}\left[\Gamma^{T_{1},T_{0}}(f^{[1]})\right](P_{n}V)-P_{n}\left[\Gamma^{T_{1,n},T_{0}}(f^{[1]})\right](P_{n}V),
\end{align}
and
\begin{align}\label{Approx-Thm-Concase-R4}
\text{\normalfont R.H.S. of \eqref{Approx-Thm-Concase-R2}}\mathrel{\mathop{\longrightarrow}^{\mathrm{\|\cdot\|_{2}}}_{\mathrm{k\to+\infty}}}\left[\Gamma^{T_{1},T_{1,n},T_{0}}(f^{[2]})\right]((T_{1}-T_{1,n})P_{n},P_{n}V).
\end{align}
Hence, by \eqref{Approx-Thm-Concase-R3} and \eqref{Approx-Thm-Concase-R4},
\begin{align}\label{Approx-Thm-Concase-R5}
\left[\Gamma^{T_{1},T_{0}}(f^{[1]})\right](P_{n}V)-P_{n}\left[\Gamma^{T_{1,n},T_{0}}(f^{[1]})\right](P_{n}V)=\left[\Gamma^{T_{1},T_{1,n},T_{0}}(f^{[2]})\right]((T_{1}-T_{1,n})P_{n},P_{n}V).
\end{align}
Using the same argument from the equality \eqref{Lem-Prjout-Cont-Eq2}, we obtain
\begin{align}
\label{Approx-Thm-Concase-R6}\left[\Gamma^{T_{1,n},T_{0}}(f^{[1]})\right](P_{n}V)-P_{n}\left[\Gamma^{T_{1,n},T_{0,n}}(f^{[1]})\right](P_{n}V)=\left[\Gamma^{T_{1,n},T_{0},T_{0,n}}(f^{[2]})\right](P_{n}V,(T_{0}-T_{0,n})P_{n})
\end{align}
for given $f$. Next, we got from Proposition \ref{Perturbation-Formula} 
\begin{align}
\nonumber&\mathcal{R}_{2}^{\lin}(f,T_{0},V)-P_{n}\left\{f(T_{1,n})-f(T_{0,n})-\frac{d}{ds}\bigg|_{s=0}f(T_{s,n})\right\}\\
\nonumber&=\left(\left[\Gamma^{T_{1},T_{0}}(f^{[1]})\right](V)-P_{n}\left[\Gamma^{T_{1,n},T_{0,n}}(f^{[1]})\right](V_{n})\right)\\
\nonumber&\hspace*{1cm}-\left(\left[\Gamma^{T_{0},T_{0}}(f^{[1]})\right](V)-P_{n}\left[\Gamma^{T_{0,n},T_{0,n}}(f^{[1]})\right](V_{n})\right)\\
\nonumber&=:\Gamma_{n}^{1,0}-\Gamma_{n}^{0,0}.
\end{align}
Again, using the identity $P_{n}+P_{n}^{\perp}=I$, \eqref{Approx-Thm-Concase-R5} and \eqref{Approx-Thm-Concase-R6} we derive
\begin{align}
\nonumber\Gamma_{n}^{1,0}&=\left[\Gamma^{T_{1},T_{0}}(f^{[1]})\right](P_{n}V)-P_{n}\left[\Gamma^{T_{1,n},T_{0}}(f^{[1]})\right](P_{n}V)+\left[\Gamma^{T_{1},T_{0}}(f^{[1]})\right](P_{n}^{\perp}V)\\
\nonumber&\quad+P_{n}\left[\Gamma^{T_{1,n},T_{0}}(f^{[1]})\right](V_{n})-P_{n}\left[\Gamma^{T_{1,n},T_{0,n}}(f^{[1]})\right](V_{n})+P_{n}\left[\Gamma^{T_{1,n},T_{0}}(f^{[1]})\right](P_{n}VP_{n}^{\perp})\\
\nonumber&=\left[\Gamma^{T_{1},T_{1,n},T_{0}}(f^{[2]})\right](\left(T_{1}-T_{1,n}\right)P_{n},P_{n}V)+P_{n}\left[\Gamma^{T_{1,n},T_{0},T_{0,n}}(f^{[2]})\right]\left(V_{n},\left(T_{0}-T_{0,n}\right)P_{n}\right)\\
\label{Approx-Thm-Concase-R7}&\quad+\left[\Gamma^{T_{1},T_{0}}(f^{[1]})\right](P_{n}^{\perp}V)+P_{n}\left[\Gamma^{T_{1,n},T_{0}}(f^{[1]})\right](P_{n}VP_{n}^{\perp}).
\end{align}
Similarly
\begin{align}
\nonumber\Gamma_{n}^{0,0}&=\left[\Gamma^{T_{0},T_{0,n},T_{0}}(f^{[2]})\right](\left(T_{0}-T_{0,n}\right)P_{n},P_{n}V)+P_{n}\left[\Gamma^{T_{0,n},T_{0},T_{0,n}}(f^{[2]})\right]\left(V_{n},\left(T_{0}-T_{0,n}\right)P_{n}\right)\\
\label{Approx-Thm-Concase-R8}&\quad+\left[\Gamma^{T_{0},T_{0}}(f^{[1]})\right](P_{n}^{\perp}V)+P_{n}\left[\Gamma^{T_{0,n},T_{0}}(f^{[1]})\right](P_{n}VP_{n}^{\perp}).
\end{align}
Combining \eqref{Approx-Thm-Concase-R7} and \eqref{Approx-Thm-Concase-R8}, and applying Proposition \ref{Perturbation-Formula}, we obtain
\begin{align}
\nonumber\Gamma&_{n}^{1,0}-\Gamma_{n}^{0,0}\\
\nonumber&=\left[\Gamma^{T_{1},T_{1,n},T_{0}}(f^{[2]})\right](\left(T_{1}-T_{1,n}\right)P_{n},P_{n}V)+P_{n}\left[\Gamma^{T_{1,n},T_{0},T_{0,n}}(f^{[2]})\right]\left(V_{n},\left(T_{0}-T_{0,n}\right)P_{n}\right)\\
\nonumber&\quad-\left[\Gamma^{T_{0},T_{0,n},T_{0}}(f^{[2]})\right](\left(T_{0}-T_{0,n}\right)P_{n},P_{n}V)-P_{n}\left[\Gamma^{T_{0,n},T_{0},T_{0,n}}(f^{[2]})\right]\left(V_{n},\left(T_{0}-T_{0,n}\right)P_{n}\right)\\
\label{Approx-Thm-Concase-R9}&\quad+\left[\Gamma^{T_{1},T_{0},T_{0}}(f^{[2]})\right](V,P_{n}^{\perp}V)+P_{n}\left[\Gamma^{T_{1,n},T_{0,n},T_{0}}(f^{[2]})\right](V_{n},P_{n}VP_{n}^{\perp}).
\end{align}
Note that $\left(T_{0}-T_{0,n}\right)P_{n}=P_{n}^{\perp}T_{0}P_{n}$ and $\left(T_{1}-T_{1,n}\right)P_{n}=P_{n}^{\perp}T_{1}P_{n}$. Therefore, application of Lemma \ref{appthm} together with \eqref{Approx-Thm-Concase-R9} and \eqref{Approx-Thm-Concase-R1} gives
\begin{align*}
\bigg\|\mathcal{R}_{2}^{\lin}(f,T_{0},V)-P_{n}\left\{f(T_{1,n})-f(T_{0,n})-\frac{d}{ds}\bigg|_{s=0}f(T_{s,n})\right\}\bigg\|_{1}\underset{n\to+\infty}{\longrightarrow}0.
\end{align*}
This completes the proof.
\end{proof}

The proof of Theorem \ref{Approx-Thm-Concase} highlights the crucial role of the identities \eqref{Approx-Thm-Concase-R5} and \eqref{Approx-Thm-Concase-R6}, which originate from Lemma \ref{Lem-Prjout-Cont}. The reader familiar with the perturbation formula \eqref{Perturbation-Formula-Eq} may easily recognize these identities as similar to \eqref{Perturbation-Formula-Eq} and may be tempted to apply the formula \eqref{Perturbation-Formula-Eq} rather than the latter identities. However, despite their similarity, \eqref{Perturbation-Formula-Eq} cannot be used directly to prove Theorem \ref{Approx-Thm-Concase}, because the operators $T_{0}-T_{0,n}$ and $T_{1}-T_{1,n}$ are not necessarily elements of $\Sp^{2}(\hilh)$.

\begin{proof}[Proof of Theorem \ref{SOT-Concase1}]
Denote $T_{s}=T_{0}+sV,~s\in[0,1]$. By Theorems \ref{SOT-Concase-FDH} and \ref{Approx-Thm-Concase} we derive
\begin{align}
\nonumber\Tr\left\{\mathcal{R}_{2}^{\lin}(f,T_{0},V)\right\}&=\lim_{n\to+\infty}\Tr\left\{P_{n}\left(f(T_{1,n})-f(T_{0,n})-\frac{d}{ds}\bigg|_{s=0}f(T_{s,n})\right)P_{n}\right\}\\
\nonumber&=\lim_{n\to+\infty}\int_{\cir}f''(z)\eta_{n}(z)dz,
\end{align}
with 
\[\eta_{n}(z)=\int_{0}^{1}\operatorname{Tr} \Big[V_{n}\Big\{\mathcal{E}_{0,n}(Arg(z))-\mathcal{E}_{s,n}(Arg(z))\Big\}\Big]ds,\]
where $\mathcal{E}_{s,n}(\cdot)$ is the semi-spectral measure corresponding to the contraction $T_{s,n}$. 

For $p=1,\infty$, consider
\[H^{p}(\cir):=\{f\in L^{p}(\cir) \mid \widehat{f}(n)=0,~\text{for all}~n<0\}.\]
Then, $H^{\infty}(\cir)$ is isometrically isomorphic to the dual of the factor-space $\left(L^{1}(\cir)/H^{1}(\cir)\right)$. Moreover, for $f\in L^{1}(\cir)$ the $L^{1}(\cir)/H^{1}(\cir)$-norm of $\dot{f}$ is defined by
\begin{align}
\label{SOT-Concase1-R1}\big\|\dot{f}\big\|_{L^{1}(\cir)/H^{1}(\cir)}=\sup_{\substack{g\in H^{\infty}(\cir)\\\|g\|_{\infty,\cir}\le1}}\left|\int_{\cir}g(z)f(z)dz\right|.
\end{align}
Moreover, the equality \eqref{SOT-Concase1-R1} holds when $g$ runs over the set of all complex polynomials $\mathcal{P}(\cir)$ with $\|g\|_{\infty,\cir}\le1$ (see \cite[Lemma 5]{PoSu12}). 

Now our remaining task is to show the convergence of $\left\{\dot{\eta_{n}}\right\}_{n\ge1}\subset L^{1}(\cir)/H^{1}(\cir)$ to $\dot{\eta}$ in $L^{1}(\cir)/H^{1}(\cir)$-norm, for some $\eta\in L^{1}(\cir)$, which now readily follows from \cite[Theorem 5.1]{ChDaPr24}. This completes the proof.
\end{proof}

\subsubsection{Koplienko-Neidhardt trace formula for a pair of contractions with one of being a strict contraction}\hfil

The main theorem of this section is the following.

\begin{thm}\label{SOT-Concase2}
Let $T_{1},T_{0}\in\conth$ such that $\|T_{0}\|<1$ and $V:=T_{1}-T_{0}\in\Sp^{2}(\hilh)$. Let $f\in\A(\D)$ be such that $f',f''\in\A(\D)$ and $f_{\cir}^{[2]}$ admits a factorization \eqref{facto2bis}. Then, there exists a function $\eta:=\eta_{T_{0},V}\in L^1(\cir)$ such that
\begin{align*}
\Tr\left\{\mathcal{R}_{2}^{\lin}(f,T_{0},V)\right\}=\int_{\cir}f''(z)\eta(z)dz.
\end{align*}
\end{thm}

\begin{proof}
Let $U_{0}$ and $U_{1}$ be the Sch\"{a}ffer matrix dilations for $T_{0}$ and $T_{1}$ respectively on the same Hilbert space $\hilk:=\ell_{2}(\hilh)\oplus \hilh\oplus\ell_{2}(\hilh)$ (see \eqref{Schaffer-Dil}). Then, using \cite[Theorem 2.5]{BhChGiPr23} and the fact that $\|T_{0}\|<1$, we get
$U_{1}-U_{0}\in\Sp^{2}(\hilk)$. Denote $U_{t}=U_{0}+t(U_{1}-U_{0})$, $t\in[0,1]$. The assumption on $f$ ensures that $f,f'$ has absolutely convergent power series on $\overline{\D}$,
\[f(z)=\sum_{n=0}^{\infty}a_{n}z^{n},\quad\text{where}~a_{n}=\frac{f^{(n)}(0)}{n!}=\widehat{f}(n).\]
Then, by examining the block representation of $U_{1}^{k}-U_{0}^{k}$ and $\frac{d}{ds}\big|_{s=0}U_{s}^{k}$ for $k\in\N$, in the similar spirit as in the proof of \cite[Theorem 3.5]{ChPrSk24} we confirm the block representation for the remainder
\begin{align*}
\mathcal{R}_{2}^{\lin}(f,U_{0},U_{1}-U_{0})=\begin{blockarray}{ccccc}
\ell_2(\hilh)&\hilh&\ell_2(\hilh)& \\[4pt]
\begin{block}{[ccc]cc}
0&0&0&\ell_{2}(\hilh) \\[3pt]
\Conv&\mathcal{R}_{2}^{\lin}(f,T_{0},V)&0&\hilh\\[3pt]
\Conv&\Conv&0&\ell_{2}(\hilh)\\[3pt]
\end{block}
\end{blockarray},
\end{align*}	
where \enquote*{$\Conv$} denotes a non-zero entry of the matrix. By Theorem \ref{Approx-Thm-Concase}, we have that $\mathcal{R}_{2}^{\lin}(f,U_{0},U_{1}-U_{0})\in\Sp^{1}(\hilk)$, which implies
\begin{align*}
\mathcal{R}_{2}^{\lin}(f,T_{0},V)\in\Sp^{1}(\hilh).
\end{align*} 
On the other hand \cite[Lemma 3.3]{ChPrSk24} implies
\[\Tr\left\{\mathcal{R}_{2}^{\lin}(f,U_{0},U_{1}-U_{0})\right\}=\Tr\left\{\mathcal{R}_{2}^{\lin}(f,T_{0},V)\right\}.\]
Hence, by Theorem \ref{SOT-Concase1} there exists $\eta\in L^{1}(\cir)$ such that
\begin{align*}
\Tr\left\{\mathcal{R}_{2}^{\lin}(f,T_{0},V)\right\}=\Tr\left\{\mathcal{R}_{2}^{\lin}(f,U_{0},U_{1}-U_{0})\right\}=\int_{\cir}f''(z)\eta(z)dz.
\end{align*}
This completes the proof of the theorem.
\end{proof}

\subsubsection{Koplienko-Neidhardt trace formula for a pair of contractions via multiplicative path}\hfil

For a bounded operator $T$, we set
$T^{*}:=T^{-1}$. Let $f(z)=\sum_{k\in\Z} \widehat{f}(k)z^{k}$ such that $\sum_{k\in\Z}\big|k\widehat{f}(k)\big|<\infty$. For $T\in\conth$, define
\begin{align}
\label{Cont-Const}f(T):=\sum_{k=0}^{\infty}\widehat{f}(k)T^{k}+\sum_{k=1}^{\infty}\widehat{f}(-k) (T^{*})^{k}=\sum_{k\in\Z}\widehat{f}(k)T^{k},
\end{align}
where the series converges absolutely in operator norm. The functions of contractions given by \eqref{Cont-Const} were initially considered in \cite{Ne88-II}.

In Theorem \ref{SOT-Concase-M} below we establish the Koplienko-Neidhardt trace formula for a pair of contractions using multiplicative path, by enlarging the admissible function class considered in \cite[Theorem 3.5]{ChPrSk24}.

\begin{thm}\label{SOT-Concase-M}
Let $T_{0}\in\conth$ and $A\in\Spsa^{2}(\hilh)$. Denote $T_{s}=e^{isA}T_{0}, ~s\in[0,1]$. Let $f:\cir\to\C$ be a twice differentiable function, and assume that $f^{[2]}$ admits a factorization \eqref{facto2bis}. Then, there exist functions $\eta_{1},\eta_{2}\in L^{1}(\cir)$, associated with the pair $\{T_{0},A\}$, such that
\begin{align*}
\Tr\left\{\mathcal{R}_{2}^{\mult}(f,T_{0},A)\right\}=\sum_{k=1}^{2}\int_{\cir}f^{(k)}(z)\eta_{k}(z)dz.
\end{align*}
\end{thm}

\begin{proof}
Let $U_{0}$ be the Sch\"{a}ffer matrix unitary dilation (see \eqref{Schaffer-Dil}) of $T_0$ on the Hilbert space $\mathcal{K}=\ell_{2}(\hilh)\oplus \hilh\oplus\ell_{2}(\hilh)$. Let
$$B=\begin{bmatrix}
0& 0& 0 \\
0& A& 0 \\
0& 0& 0
\end{bmatrix}:\mathcal{K}\to\hilk$$ 
be the self-adjoint extension of $A$ on $\mathcal{K}$. Then $B\in\Sp^{2}(\mathcal{K})$ and $\|A\|_{2}=\|B\|_{2}$. Consequently $U_{s}=e^{isB}U_{0}$ is a unitary dilation of $T_{s}$ on $\mathcal{K}$ (can be seen from the block representations of $B$ and $U_{0}$).
	
Note that the assumption on $f^{[2]}$ implies that $f''$ is bounded on $\cir$, which in turn confirms that $\sum_{k\in\Z}\big|k\widehat{f}(k)\big|<\infty$. Hence, we have
\[\frac{d}{ds}\bigg|_{s=0}f(T_{s})=\sum_{k\in\Z}\widehat{f}(k)\frac{d}{ds}\bigg|_{s=0}T_{s}^k,\]
where the series converges absolutely in the operator norm. Also,
\[\mathcal{R}_{2}^{\mult}(f,T_{0},A)=\sum_{k\in\Z}\widehat{f}(k)(e^{iA}T_{0})^k-\sum_{k\in\Z}\widehat{f}(k)T_{0}^k-\sum_{k\in\Z}\widehat{f}(k)\frac{d}{ds}\bigg|_{s=0}T_{s}^{k},\]
where each of the above series converge absolutely in the operator norm. Hence, we can express the above identity as follows:	
\begin{align*}
&\mathcal{R}_{2}^{\mult}(f,T_{0},A)=\sum_{k\in\Z}\widehat{f}(k)\mathcal{R}_{2}^{\mult}(z^{k},T_{0},A),\\
\text{ similarly, }~ &\mathcal{R}_{2}^{\mult}(f,U_{0},B)=\sum_{k\in\Z}\widehat{f}(k)\mathcal{R}_{2}^{\mult}(z^{k},U_{0},B).
\end{align*}
Next, by following the same lines of arguments as in the proof of \cite[Theorem 3.5]{ChPrSk24}, we have
\begin{align}
\label{SOT-Concase-M-R1}\mathcal{R}_{2}^{\mult}(f,T_{0},A)=P_{\hilh}\mathcal{R}_{2}^{\mult}(f,U_{0},B)\big|_{\hilh},
\end{align} 
where $P_{\hilh}$ is the orthogonal projection from $\hilk$ onto the subspace $0\oplus\hilh\oplus0$. Since $f^{[2]}$ satisfies the factorization \eqref{facto2bis}, by Theorem \ref{N-Traceformula} and \eqref{SOT-Concase-M-R1}, we have
\begin{align*}
\mathcal{R}_{2}^{\mult}(f,T_{0},A)\in\Sp^1(\hilh).
\end{align*}
Again, from \cite[Eq. (3.29) and Lemma 3.3]{ChPrSk24}, it follows that
\[\Tr\left(\mathcal{R}_{2}^{\mult}(f,T_{0},A)\right)=\Tr\left(\mathcal{R}_{2}^{\mult}(f,U_{0},B)\right).\]
Finally, using Theorem \ref{N-Traceformula}, we complete the proof.
\end{proof}

\section{Modified trace formulas}\label{Modified trace formulas}

As discussed in the introduction, an $n$-times differentiable function $f:\R\,(\text{or}\,\cir)\to\C$ with with bounded derivatives is not sufficient to make the Taylor remainder $\mathcal{R}_{n}^{\lin}(f, H_{0},V)$ (or $\mathcal{R}_{n}^{\mult}(f,U_{0},A)$) a trace class operator, where $V\in\Sp_{sa}^{n}(\hilh)$ (or $A\in\Sp_{sa}^{n}(\hilh)$). However, these conditions on $f$ guarantee that (see \cite[Proposition 3.3]{Co22} and \cite[Remark 5.3]{Co24})
\[\mathcal{R}_{n}^{\lin}(f, H_{0},\cdot),\mathcal{R}_{n}^{\mult}(f,U_{0},\cdot):\Sp_{sa}^{p}(\hilh)\mapsto\Sp^{p/n}(\hilh),\quad\text{for}~p>n.\]
In this section, we will make use of these above estimates and modify the remainder as follows:
\[\mathcal{R}_{n}^{\lin}(f,H_{0},V)\cdot X\quad\text{and}\quad \mathcal{R}_{n}^{\mult}(f,U_{0},A)\cdot X\]	
where $V,A\in\Spsa^{n+\epsilon}(\hilh)$ for $\epsilon>0$, and $X\in\Sp^{1+\frac{n}{\epsilon}}(\hilh),\epsilon>0$. This modification allows us to establish integral representations for $\Tr\left\{\mathcal{R}_{n}^{\lin}(f,H_{0},V)\cdot X\right\}$ (or $\Tr\left\{\mathcal{R}_{n}^{\mult}(f,U_{0},A)\cdot X\right\}$) for $n$-times differentiable functions $f$ on $\R$ (or $\cir$) with bounded $f^{(n)}$. This type of modification was initially considered in \cite{Sk17Adv} and later further investigated in \cite{BhChGiPr23}, where the terminology \enquote{modified trace formula} was introduced.

\subsection{Self-adjoint case}

\begin{thm}\label{Thm-Mod-Sacase}
Let $n\in\N$ and $\epsilon>0$. Let $p=n+\epsilon$ and $q=1+\frac{n}{\epsilon}$. Let $H_0,V,X$ be three self-adjoint operators in $\hilh$ such that $V\in\Sp^{p}(\hilh)$ and $X\in\Sp^{q}(\hilh)$. Let $f:\R\to\C$ be $n$-times differentiable on $\R$ with bounded $f^{(n)}$.
Then, there exists $\eta_{n}:=\eta_{n,H_0,V,X}\in L^{1}(\R)$ such that
\begin{align}\label{TMS-R1}
\Tr\left\{\left[\Gamma^{H_0+V,H_0,\ldots,H_0}(f^{[n]})\right](V,\ldots,V)\cdot X\right\}=\int_{\R}f^{(n)}(x)\eta_{n}(x)dx.
\end{align}
In particular, if either $H_0$ is bounded or $f^{(i)}$ is bounded for all $1\le i\le n-1$, then
\begin{align}\label{TMS-R2}
\Tr\left\{\mathcal{R}^{\lin}_{n}(f,H_0,V)\cdot X\right\}=\int_{\R}f^{(n)}(x)\eta_{n}(x)dx.
\end{align}
Moreover, there exists a constant $c$ (independent of $f$, $H_{0},V$ and $X$) such that $\eta_{n}$ satisfying \eqref{TMS-R1} can be chosen so that 
\[\|\eta_{n}\|_{L^{1}}\le c\,\|V\|_{p}^n\|X\|_{q}.\]
\end{thm}

\begin{proof} 
Denote $\Gamma_{n}:=\left[\Gamma^{H_0+V,H_0,\ldots,H_0}(f^{[n]})\right](V,\ldots,V)$. Assume first that $f^{(n-1)},f^{(n)}\in  C_{0}(\R)$. By Theorem \ref{MOIEst-Self+Uni},
\begin{align*}
\left|\Tr\left\{\Gamma_{n}\cdot X\right\}\right|\le c\left\|f^{(n)}\right\|_{\infty,\R}\|V\|_{p}^{n}\|X\|_{q}.
\end{align*}
By the Riesz-representation theorem for elements of $(C_{0}(\R))^*$, there exists a measure $\mu$ on $\R$ such that
\begin{align}
\label{TMS-R3}\Tr\left\{\Gamma_n\cdot X\right\}=\int_{\R}f^{(n)}(x)d\mu(x).
\end{align}
Next, we show that the measure $\mu$ is absolutely continuous with respect to the Lebesgue measure. According to \cite[Proposition 2.8]{Co22}, we have
$$\Gamma_{n}=\begin{cases}
\Gamma_{n-1}-\left[\Gamma^{H_0,\ldots,H_0}(f^{[n-1]})\right](V,\ldots,V),&n>1\\
f(H_{0}+V)-f(H_{0}),&n=1.
\end{cases}$$
Assume that $X\in\Sp^1(\hilh)$. By Theorem \ref{MOIEst-Self+Uni}, we have
\begin{align*}
\left|\Tr\left\{\Gamma_{n}\cdot X\right\}\right|\le c\left\|f^{(n-1)}\right\|_{\infty,\R}\|V\|_{p}^{n-1}\|X\|_{\frac{n+\epsilon}{1+\epsilon}}.
\end{align*}
Since $f^{(n-1)}\in C_{0}(\R)$, this yields the existence of a measure $\nu_{n,H_0,V,X}\in(C_{0}(\R))^*$ such that
$$\Tr\left\{\Gamma_{n}\cdot X\right\}=\int_{\R}f^{(n-1)}(x)d\nu_{n,H_0,V,X}(x).$$
By integration by parts, we get
$$\Tr\left\{\Gamma_{n}\cdot X\right\}=\int_{\R}f^{(n)}(x)\eta_{n,H_0,V,X}(x)dx,$$
where $\eta_{n,H_0,V,X}(x)=-\nu_{n,H_0,V,X}((-\infty,x])$. Comparing with \eqref{TMS-R3}, this gives
$$d\mu(x)=\eta_{n,H_0,V,X}(x)dx$$
and hence
$$d|\mu|(x)=|\eta_{n,H_0,V,X}(x)|dx.$$
Since $|\mu|$ is finite, we get that $\eta_{n,H_0,V,X}\in L^1(\R)$. Hence, we showed the absolute continuity of $\mu$ and
\begin{align}
\label{TMS-R4}\Tr\left\{\Gamma_n\cdot X\right\}=\int_{\R}f^{(n)}(x)\eta_{n,H_0,V,X}(x)dx.
\end{align}
Moreover, there exists a constant $c$, independent of $f,H_{0},V$, and $X$ such that
\begin{align}
\label{TMS-R5}\|\eta_{n,H_0,V,X}\|_{L^{1}}\le c\|V\|_{p}^{n}\|X\|_{q}.
\end{align}

If $X\in\Sp^{q}(\hilh)$, we let $\{X_{l}\}_{l\ge1}\subset\Sp^{1}(\hilh)$ be such that $\|X-X_{l}\|_{q}\to 0$ as $l\to+\infty$. Then, there exists a sequence of functions $\left\{\eta_{n,H_0,V,X_{l}}\right\}_{l\ge1}\subset L^{1}(\R)$ such that \eqref{TMS-R4} holds. We show that the sequence of spectral shift functions $\{\eta_{n,H_0,V,X_l}\}_{l\ge1}$ is Cauchy in $L^{1}(\R)$. For $f\in C_c^{n+1}(\R)$, we have
\begin{align}
\nonumber&\left|\int_{\R}f^{(n)}(x)\eta_{n,H_0,V,X_l}(x)dx-\int_{\R}f^{(n)}(x)\eta_{n,H_0,V,X_m}(x)dx\right|\\
\nonumber&=\left|\Tr\left\{\Gamma_n\cdot \left(X_{l}-X_{m}\right)\right\}\right|\le c\left\|f^{(n)}\right\|_{\infty,\R}\|V\|_{p}^{n}\|X_{l}-X_{m}\|_{q}\\
\nonumber&\hspace*{1.6in}\underset{l,m\to+\infty}\longrightarrow 0.	
\end{align}
Thus, there exists an integrable function $\eta_{n,H_0,V,X}$ such that $\|\eta_{n,H_0,V,X_{l}}-\eta_{n,H_0,V,X}\|_{L^{1}}\to0$ as $l\to+\infty$. Thus
\begin{align*}
\Tr\left\{\Gamma_{n}\cdot X\right\}
&=\lim_{l\to+\infty}\int_{\R}f^{(n)}(x)\eta_{n,H_0,V,X_{l}}(x)dx\\
&=\int_{\R}f^{(n)}(x)\eta_{n,H_0,V,X}(x)dx.
\end{align*}
Applying \eqref{TMS-R5} to $\eta_{n,H_{0},V,X_{l}}$ for every $l\in\N$, we deduce the bound
\[\|\eta_{n,H_{0},V,X}\|_{L^{1}}\le c\|V\|_{p}^{n}\|X\|_{q}.\]
This completes the proof in the case $f^{(n-1)},f^{(n)} \in C_{0}(\R)$.

Let us extend \eqref{TMS-R1} to the case $f\in C^{n}(\R)$ with bounded $f^{(n)}$. This implies that there exist $a,b\ge 0$ such that for every $x\in\R$, $|f^{(n-1)}(x)|\le a|x|+b$. Let, for every $k\in\N$, $g_{k}\in C^n(\R)$ be such that $$g_{k}^{(n-1)}(x)=\exp(-\frac{x^2}{k})f^{(n-1)}(x).$$ In particular, $g_{k}^{(n-1)}\in C_0(\R)$. Since
$$g_{k}^{(n)}(x)=-\frac{2x}{k}\exp(-\frac{x^2}{k})f^{(n-1)}(x)+\exp(-\frac{x^2}{k}) f^{(n)}(x),$$
we see that $g_{k}^{(n)} \in C_0(\R)$, $g_{k}^{(n)} \to f^{(n)}$ pointwise on $\R$ as $k\to+\infty$, and
it is easy to check that 
\begin{align*}
|g_{k}^{(n)}(x)|&\le 2(a|x|+b) \frac{|x|}{k}\exp(-\frac{x^2}{k})+\|f^{(n)}\|_{\infty,\R}\\
&\le(2a+\sqrt{2}b)\exp(-1/2)+\|f^{(n)}\|_{\infty,\R}:=M.
\end{align*}
In particular, it follows from \cite[Formula (7.12)]{DeLo93} that $g_{k}^{[n]}\to f^{[n]}$ pointwise. According to Theorem \ref{MOIEst-Self+Uni} and \cite[Lemma 2.3]{Co22}, we deduce that
$$\Tr\left\{\left[\Gamma^{H_0+V,H_0,\ldots,H_0}(g_k^{[n]})\right](V,\ldots,V)\cdot X\right\}\underset{k\to+\infty}{\longrightarrow}\Tr\left\{\left[\Gamma^{H_0+V,H_0,\ldots,H_0}(f^{[n]})\right](V,\ldots,V) \cdot X\right\}.$$
Applying the latter and Lebesgue's dominated convergence theorem in the equality
\begin{align*}
\Tr\left\{ \left[\Gamma^{H_0+V,H_0,\ldots,H_0}(g_k^{[n]})\right](V,\ldots,V) \cdot X\right\} =\int_{\R}g_k^{(n)}(x)\eta_{n,H_0,V,X}(x)dx
\end{align*}
yield \eqref{TMS-R1}.

The extension to the case when $f$ is only $n$-times differentiable and $f^{(n)}$ is bounded follows from similar arguments by considering a sequence $\{h_{k}\}_{k\ge1}\subset C^n(\R)$ such that 
$$h_{k}^{(n)}\to f^{(n)}~\text{pointwise on}~\R\quad\text{and}\quad\|h_{k}^{(n)}\|_{\infty,\R}\le M.$$
For instance any $h_{k}\in C^n(\R)$ such that $$h_{k}^{(n)}(x)=k(f^{(n-1)}(x+1/k)-f^{(n-1)}(x)),~x\in\R$$ 
satisfy these conditions.

Finally, the identity \eqref{TMS-R2} follows from \cite[Proposition 3.3]{Co22}. This concludes the proof.
\end{proof}

\begin{rmrk}
\begin{enumerate}[{\normalfont(i)}]
\item For a finite-dimensional Hilbert space $\hilh$, formula \eqref{TMS-R2} establishes the classical trace formula by setting $X = I$ for a large class of admissible functions. For infinite dimensional Hilbert space $\hilh$, using the trace formula \eqref{TMS-R2}, we can analyze the spectral data of each principal block of $\mathcal{R}^{\lin}_{n}(f,H_0,V)$ in terms of the trace. Indeed, consider the perturbation $V\in\Spsa^{n+\epsilon}(\hilh)$ with  $\epsilon>0$, and let $X:=P_{N}$ be the orthogonal projection onto $\{e_{1},\ldots,e_{N}\}$, where $\{e_{i}\}_{i=1}^\infty$ is an orthonormal basis for $\hilh$. Then
\begin{align*}
\left|\Tr\left\{\mathcal{R}^{\lin}_{n}(f,H_0,V) \cdot X\right\}\right|\le c\,N^{\frac{\epsilon}{n+\epsilon}} \|f^{(n)}\|_{\infty,\R}\,\|V\|_{n+\epsilon}^n.
\end{align*}
Hence, as $\epsilon$ approaches to $0$, the trace of the block depends on $V$ only.
\vspace*{0.1cm}
\item If $\mathcal{R}^{\lin}_{n}(f,H_0,V)=0$, then $\eta_{n,H_0,V,X}=0$ for all $X\in\Sp^{1+\frac{n}{\epsilon}}(\hilh)$. The converse is also true. Indeed, suppose for each $X\in\Sp^{1+\frac{n}{\epsilon}}(\hilh)$, $\eta_{n,H_0,V,X}=0$, then $\Tr\left\{\mathcal{R}^{\lin}_{n}(f,H_0,V)\cdot X\right\}=0$. Now by duality, we conclude that $\mathcal{R}^{\lin}_{n}(f,H_0,V)=0$.
\end{enumerate}
\end{rmrk}

\subsection{Unitary case}

\begin{thm}\label{Thm-Mod-Unicase}
Let $n\in\N$ and $\epsilon>0$. Let $p=n+\epsilon$ and $q=1+\frac{n}{\epsilon}$. Let $U_0\in\mathcal{U}(\hilh)$, $A\in\Spsa^{p}(\hilh)$ and $X\in\Sp^{q}(\hilh)$. Let $f:\cir\to\C$ be $n$-times differentiable on $\cir$ with bounded $f^{(n)}$. Then, there exist functions $\xi_{n}:=\xi_{n,U_0,A,X}\in L^{1}([0,2\pi])$ and $\xi_{k}:=\xi_{k,U_0,A,X}\in\text{\normalfont span}\left\{1,\ldots,e^{i(k+1-n)t}\right\},~k=1,\ldots,n-1$ such that
\begin{align}\label{TMU-R1}
\Tr\left\{\mathcal{R}^{\mult}_{n}(f,U_0,A)\cdot X\right\}=\sum_{k=1}^{n}\int_{0}^{2\pi}f^{(k)}(e^{it})\xi_{k}(t)dt
\end{align}
holds. Additionally, if $\widehat{f}(k)=0$ for $k=1,\ldots,n-1$, then \eqref{TMU-R1} holds with $\xi_{k}=0$ for $k=1,\ldots,n-1$.
\end{thm}
	
\begin{proof}
Let
\[\mathfrak{F}:=\Big\{f\in C(\cir)\mid  \text{$f$ is $n$-times differentiable on $\cir$ with bounded $f^{(n)}$}\Big\}.\]
Our first step is to establish \eqref{TMU-R1} for $f\in C^{n}(\cir)$, then we will extend it to $f\in\mathfrak{F}$ through approximations. By following arguments similar to those in \cite[Theorems 3.1 and 3.2]{BhChGiPr23}, we establish \eqref{TMU-R1} for $n=1$ and functions $f\in C^{1}(\mathbb{T})$. Assume now that $n\geq 2$.

We have $C^{n}(\cir)=\mathcal{P}_{n}(\cir)+\mathcal{G}_{n}(\cir)$, where
\[\mathcal{G}_{n}(\cir):=\left\{f\in C^{n}(\cir)\mid \widehat{f}(k)=0~\text{for}~k=1,\ldots,n-1\right\},\]
and $\mathcal{P}_{n}(\cir)$ is the set of polynomials of degree at most $n-1$ on $\cir$.

Note that the existence of shift functions $\xi_{k}$ for $k=1,\ldots,n-1$ depends on the class $\mathcal{P}_{n}(\cir)$, which follows from a closer investigation of the proof of \cite[Theorem 3.2]{ChPrSk24}.

Below we present the proof of \eqref{TMU-R1} with $\xi_{1}=\cdots=\xi_{n-1}=0$ for $f\in\mathcal{G}_{n}(\cir)$. Take $f\in\mathcal{G}_{n}(\cir)$. Denote 
\[\mathcal{R}_{k,n}:=\sum_{\substack{l_{1},\ldots,l_{k}\ge1\\l_{1}+\cdots+l_{k}=n}}\left[\Gamma^{e^{iA}U_{0},(U_{0})^{k}}(f^{[k]})\right]\bigg(\sum_{r=l_{1}}^{\infty}\frac{(iA)^{r}}{r!}U_{0},\frac{(iA)^{l_{2}}}{l_{2}!}U_{0},\ldots,\frac{(iA)^{l_{k}}}{l_{k}!}U_{0}\bigg)\]
for $k=1,\ldots,n.$ By noting that $\sum_{r=l_{1}}^{\infty}\frac{(iA)^{r}}{r!}U_{0}\in\Sp^{p/l_{1}}(\hilh)$, the assumption on $f$, together with Theorem \ref{MOIEst-Self+Uni}, ensures the existence of a constant $c_{p,n}>0$ such that
\begin{align}
\label{TMU-R2}&\big|\Tr\big\{\mathcal{R}_{k,n}\cdot X\big\}\big|\le c_{p,n}\left\|f^{(k)}\right\|_{\infty,\cir}\|A\|_{p}^{n}\|X\|_{q},\quad\text{for}~k=1,\ldots,n.
\end{align}
Now observing the expression of $\mathcal{R}^{\mult}_{n}(f,U_0,A)$ (see, e.g., \cite[Eq. (26)]{ChCoGiPr24}), the estimate \eqref{TMU-R2}, and lastly using the Riesz-representation theorem for elements of $(C(\cir))^*$ we get the existence of measures $\nu_{k}$ for $1\le k\le n$ on $\cir$ such that
\begin{align}
\label{TMU-R3}\Tr\left\{\mathcal{R}^{\mult}_{n}(f,U_0,A)\cdot X\right\}=\sum_{k=1}^{n}\int_{\cir}f^{(k)}(z)d\nu_{k}(z)
\end{align}
with $\|\nu_{k}\|\le c_{p,n}\|A\|_{p}^{n}\|X\|_{q}$, $1\le k\le n$. Using integration by parts in \eqref{TMU-R3} and the fact that $f\in\mathcal{G}_{n}(\cir)$, we obtain
\begin{align}
\label{TMU-R4}\Tr\left\{\mathcal{R}^{\mult}_{n}(f,U_0,A)\cdot X\right\}=\int_{\cir}f^{(n)}(z)\zeta_{n,U_{0},A,X}(z)dz+\int_{\cir}f^{(n)}(z)d\nu_{n}(z),
\end{align}
where $\zeta_{n,U_{0},A,X}\in L^{1}(\cir)$ with $\|\zeta_{n,U_{0},A,X}\|_{L^{1}}\le \widetilde{c}_{p,n}\|A\|_{p}^{n}\|X\|_{q}$ for some constant $\widetilde{c}_{p,n}$ (for further details, see the proof of \cite[Lemma 4.5]{PoSkSu16}). Thus, \eqref{TMU-R4} leads to the following crucial estimate, which is independent of $\{f',\ldots,f^{(n-1)}\}$, such that
\begin{align}
\label{TMU-R5}\left|\Tr\left\{\mathcal{R}^{\mult}_{n}(f,U_0,A)\cdot X\right\}\right|\le d_{p,n}\|f^{(n)}\|_{\infty,\cir}\|A\|_{p}^{n}\|X\|_{q},
\end{align}
for some constant $d_{p,n}$.

Next, we got
\begin{align}
\nonumber&\mathcal{R}^{\mult}_{n}(f,U_0,A)=\mathcal{R}^{\mult}_{n-1}(f,U_0,A)-\sum_{k=1}^{n-1}\mathfrak{R}_{k},
\end{align}
where
\begin{align*}
\mathfrak{R}_{k}:=\sum_{\substack{l_{1},\ldots,l_{k}\ge 1\\l_{1}+\cdots+l_{k}=n-1}}\frac{i^{n-1}}{l_{1}!\cdots l_{k}!}\left[\Gamma^{(U_0)^{k+1}}(f^{[k]})\right](A^{l_{1}}U_0,\ldots,A^{l_{k}}U_0).
\end{align*}
Let $X\in\Sp^{1}(\hilh)$. Then, Theorem \ref{MOIEst-Self+Uni} implies
\begin{align}
\label{TMU-R6}\left|\Tr\left\{\left(\mathcal{R}_{k,n-1}-\mathfrak{R}_{k}\right)\cdot X\right\}\right|\le\widetilde{d}_{p,n}\left\|f^{(k)}\right\|_{\infty,\cir}\|A\|_{p}^{n-1}\|X\|_{\frac{n+\epsilon}{1+\epsilon}},\quad\text{for}~k=1,\ldots,n-1,	
\end{align}
where $\widetilde{d}_{p,n}$ is some constant. Similarly, by examining the expressions of $\mathcal{R}^{\mult}_{n-1}(f,U_0,A)$, and using the estimate \eqref{TMU-R6}, we get measures $\mu_{k}$ for $1\le k\le n-1$ on $\cir$ such that
\begin{align*}
\Tr\left\{\mathcal{R}^{\mult}_{n}(f,U_0,A)\cdot X\right\}=\sum_{k=1}^{n-1}\int_{\cir}f^{(k)}(z)d\mu_{k}(z).
\end{align*}
The last equation further reduces (due to integration by parts) to
\begin{align}
\label{TMU-R8}\Tr\left\{\mathcal{R}^{\mult}_{n}(f,U_0,A)\cdot X\right\}=\int_{\cir}f^{(n)}(z)\eta_{n,U_0,A,X}(z)dz,
\end{align}
where $\eta_{n,U_0,A,X}\in L^{1}(\cir)$.

If $X\in\Sp^{q}(\hilh)$, we let $\{X_{l}\}_{l\ge1}\subset\Sp^{1}(\hilh)$ be such that $\|X_{l}-X\|_{q}\to 0$ as $l\to+\infty$. Then, by \eqref{TMU-R8} there exists a sequence $\{\eta_{n,U_0,A,X_{l}}\}_{l\ge1}\subset L^{1}(\cir)$ such that
\begin{align}
\label{TMU-R9}\Tr\left\{\mathcal{R}^{\mult}_{n}(f,U_0,A)\cdot X_{l}\right\}=\int_{\cir}f^{(n)}(z)\eta_{n,U_0,A,X_{l}}(z)dz.
\end{align}
Using \cite[Lemma 4.3]{Sk17Adv} we have $\left(L^{1}(\cir)/\mathcal{P}_{n}(\cir)\right)^*\cong\mathcal{P}_{n}(\cir)^{\perp}$, where the annihilator of $\mathcal{P}_{n}(\cir)$ in $(L^{1}(\cir))^*\cong L^{\infty}(\cir)$ is 
\[\mathcal{P}_{n}(\cir)^{\perp}=\left\{h\in L^{\infty}(\cir) \mid \widehat{h}(j)=0~\text{for}~j=-n,\ldots,-1\right\}.\]
Since \[\left\{f^{(n)} \mid f\in\mathcal{G}_{n}(\cir)\right\}\subset\mathcal{P}_{n}(\cir)^{\perp},\]
we have for $\eta\in L^{1}(\cir)$
\begin{align}
\label{TMU-R10}\|\dot{\eta}\|_{L^{1}(\cir)/\mathcal{P}_{n}(\cir)}=\sup_{\substack{f\in\mathcal{G}_{n}(\cir)\\\|f^{(n)}\|_{\infty,\cir}\le1}}\left|\int_{\cir}f^{(n)}(z)\eta(z)dz\right|.
\end{align}
Hence, applying \eqref{TMU-R10} we obtain
\begin{align*}
\|\dot{\eta}_{n,U_0,A,X_{l}}-\dot{\eta}_{n,U_0,A,X_{m}}\|_{L^{1}(\cir)/\mathcal{P}_{n}(\cir)}&=\sup_{\substack{f\in\mathcal{G}_{n}(\cir)\\\|f^{(n)}\|_{\infty,\cir}\le1}}\left|\Tr\left\{\mathcal{R}^{\mult}_{n}(f,U_0,A)\cdot\left(X_{l}-X_{m}\right)\right\}\right|\\
&\overset{\eqref{TMU-R5}}{\le} d_{p,n}\|A\|_{p}^{n}\|X_{l}-X_{m}\|_{q}\\
&\underset{l,m\to+\infty}{\longrightarrow0},
\end{align*}
which shows that $\left\{\dot{\eta}_{n,U_0,A,X_{l}}\right\}_{l\ge1}$ is Cauchy in $L^{1}(\cir)/\mathcal{P}_{n}(\cir)$. Let $\dot{\eta}_{n,U_0,A,X}\in L^{1}(\cir)/\mathcal{P}_{n}(\cir)$ be the limit of $\left\{\dot{\eta}_{n,U_0,A,X_{l}}\right\}_{l\ge1}$ for $\eta_{n,U_0,A,X}\in L^{1}(\cir)$. Then by passing the limit in \eqref{TMU-R9} we have
\begin{align}
\label{TMU-R11}\Tr\left\{\mathcal{R}^{\mult}_{n}(f,U_0,A)\cdot X\right\}=\int_{\cir}f^{(n)}(z)\eta_{n,U_0,A,X}(z)dz.
\end{align}
Comparing \eqref{TMU-R4} and \eqref{TMU-R11} we prove that the measure $\nu_{n}$ is absolutely continuous with respect to Lebesgue measure and 
\begin{align*}
\Tr\left\{\mathcal{R}^{\mult}_{n}(f,U_0,A)\cdot X\right\}=\int_{0}^{2\pi}f^{(n)}(e^{it})\xi_{n}(t)dt,
\end{align*}
where \[\xi_{n}(t):=ie^{it}\left[\eta_{n,U_0,A,X}(e^{it})-\zeta_{n,U_0,A,X}(e^{it})\right].\] 
This establishes \eqref{TMU-R1} for $f\in C^{n}(\cir)$.  

Finally, consider $f\in\mathfrak{F}$. Then, there exists a sequence of functions $\{f_{N}\}_{N\ge1}\subset C^{n}(\cir)$ satisfying $\|f_{N}^{(k)}\|_{\infty,\cir}\le M$ for $k=1,\ldots,n$ such that
$$f_{N}^{(k)} \underset{N\to+\infty}{\longrightarrow} f^{(k)} \ \text{uniformly}\ \cir,\text{ for }k=1,\ldots,n-1,$$
and
$$f_{N}^{(n)}\underset{N\to+\infty}{\longrightarrow}f^{(n)}\ \text{pointwise on}\ \cir, \quad \quad f_{N}^{[n]}\underset{N\to+\infty}{\longrightarrow}f^{[n]}\ \text{pointwise on}\ \cir^{n+1}.$$
Hence, the assumption on $\{f_{N}\}_{N\ge1}$ together with \cite[Lemma 2.2]{ChCoGiPr24} implies
\begin{align}
\label{TMU-R12}\Tr\left\{\mathcal{R}^{\mult}_{n}(f_{N},U_0,A)\cdot X\right\}\underset{N\to+\infty}{\longrightarrow}\Tr\left\{\mathcal{R}^{\mult}_{n}(f,U_0,A)\cdot X\right\}.
\end{align}
Using Lebesgue's dominated convergence theorem, we further have 
\begin{align}
\label{TMU-R13}\sum_{k=1}^{n}\int_{0}^{2\pi}f_{N}^{(k)}(e^{it})\xi_{k}(t)dt\underset{N\to+\infty}{\longrightarrow}\sum_{k=1}^{n}\int_{0}^{2\pi}f^{(k)}(e^{it})\xi_{k}(t)dt.
\end{align}
Thus, \eqref{TMU-R12} and \eqref{TMU-R13} establish \eqref{TMU-R1} in its full generality.
\end{proof}

\subsection{Contraction case}

\begin{thm}\label{Thm-Mod-Concase}
Let $n\in\N$ and $\epsilon>0$. Let $p=n+\epsilon$ and $q=1+\frac{n}{\epsilon}$. Let $T_{0},T_{1}\in\conth$ be such that $T_{1}-T_{0}\in\Sp^{p}(\hilh)$ and  $X\in\Sp^{q}(\hilh)$. Let $f\in\A(\D)$ be such that $f^{(k)}\in\A(\D)$, for any $1\le k\le n$. Then there exists $\zeta_{n}:=\zeta_{n,T_{1},T_{0},X}\in L^{1}(\cir)$ such that
\begin{align*}
\Tr\left\{\mathcal{R}_{n}^{\lin}(f,T_{0},T_{1}-T_{0})\cdot X\right\}=\int_{\cir}f^{(n)}(z)\zeta_{n}(z)dz
\end{align*}
holds.
\end{thm}

\begin{proof}
The proof follows along the line of the proof of Theorem \ref{Thm-Mod-Sacase}.
\end{proof}

Next, to complete this study, we explore certain properties of the spectral shift functions obtained above. We start by discussing the non-negativity of the spectral shift functions $\eta_{n,H_0,V,X}$, where $H_0\in\bh$ and $V\in\Sp^{n+\epsilon}(\hilh)$ are self-adjoint operators, and $X\in\Sp^{1+\frac{n}{\epsilon}}(\hilh)$ is the weight. Our analysis follows a similar approach to that used in \cite{Sk17Laa}. Our primary objective is to understand how $\eta_{n,H_0,V,X}$ behaves with respect to the weight $X$.

\begin{ppsn}
Let $n\in\N$ and $\epsilon>0$. Let $H_0,V,K$ be self-adjoint operators on $\hilh$ such that $H_0\in \mathcal{B}(\hilh)$, $V\in\Sp^{n+\epsilon}(\hilh)$, and $K\in\Sp^{1+\frac{n}{\epsilon}}(\hilh)$. Assume that $H_0$ and $V$ commute and $K\ge0$. Then the shift function $\eta_{n,H_0,V,V^nK}$ in Theorem \ref{Thm-Mod-Sacase} is non-negative.
\end{ppsn}

\begin{proof} 
Let $f\in C^{n}(\R)$ with bounded $f^{(n)}$. Let $X:=V^{n}K$. By integration by parts, we have
\[\Tr\left\{\mathcal{R}^{\lin}_{n}(f,H_0,V)\cdot X\right\}=n\int_{0}^{1} (1-t)^{n-1} \Tr\left\{\left[\Gamma^{H_0+tV,\ldots,H_0+tV}(f^{[n]})\right](V,\ldots,V)\cdot X\right\}dt.\]
For every $t\in [0,1]$, we have, by \cite[Lemma 3.5]{Co22},
\begin{align*}
\left[\Gamma^{H_0+tV,\ldots,H_0+tV}(f^{[n]})\right](V,\ldots,V)\cdot X
&=\frac{1}{n!}f^{(n)}(H_0+tV)V^{n}\cdot X \\
&=\frac{1}{n!}\int_{\R}f^{(n)}(s)\,dE^{H_0+tV}(s)V^{2n}K,
\end{align*}
where $E^{H_0+tV}(\cdot)$ is the spectral measure of the self-adjoint operator $H_0+tV$. Again,
\begin{align*}
\Tr\left\{\left[\Gamma^{H_0+tV,\ldots,H_0+tV}(f^{[n]})\right](V,\ldots,V)\cdot X\right\}
&=\frac{1}{n!}\int_{\R}f^{(n)}(s)\,d\mu_{t}(s),
\end{align*}
where $\mu_{t}$ is the positive measure defined for any Borel subset $\Delta\subset\R$ by
\[\mu_{t}(\Delta)=\Tr\left(E^{H_0+tV}(\Delta)V^{2n}K\right)=\Tr\left(E^{H_0+tV}(\Delta)V^{n}KV^{n}E^{H_0+tV}(\Delta)\right).\]
As in the proof of \cite[Theorem 5.1]{CoLemSkSu19}, it makes sense to consider
\[\mu:=\frac{1}{(n-1)!}\int_{0}^{1}(1-t)^{n-1}\mu_{t}\,dt\in (C_{0}(\R))^*\]
and one can then show that
$$\Tr\left\{\mathcal{R}^{\lin}_{n}(f,H_0,V)\cdot X\right\}=\int_{\R}f^{(n)}(s)\,d\mu(s).$$
The conclusion follows from the fact that $\mu$ is a positive measure.
\end{proof}

\begin{thm}\label{Thm-Pos}
Let $n\in\N$ and $\epsilon>0$. Let $H_0=H_0^{*}\in\bh$, $V\in\Spsa^{n+\epsilon}(\hilh)$ and $X\in\Sp^{1+\frac{n}{\epsilon}}(\hilh)$. Assume $XH_0=H_0X$. Then the following assertions hold:
\begin{enumerate}[{\normalfont(A)}]	
\item If $n=2$ and $X\ge0$, then $\eta_{2,H_0,V,X}$ is non-negative.
\item Let $n>2$ and $V$ be a rank one operator. 
\begin{enumerate}[{\normalfont(i)}]
\item If $n$ is even and $X\ge0$, then $\eta_{n,H_0,V,X}\ge0$.
\item Let $n$ be odd and $X\ge0$. If $V\ge0$ $(\text{respectively},~V\le0)$, then $\eta_{n,H_0,V,X}\ge0$ $(\text{respectively,}~\eta_{n,H_0,V,X}\le0)$.
\end{enumerate}
\end{enumerate}	
\end{thm}

\begin{proof}
According to Theorem \ref{Thm-Mod-Sacase} for $n$-times differentiable $f$ with bounded $f^{(n)}$, there exists $\eta_{n,H_0,V,X}\in L^{1}(\R)$ such that
\begin{align}
\nonumber\Tr\left\{\mathcal{R}^{\lin}_{n}(f,H_0,V)\cdot X\right\}&=\Tr\left\{\left[\Gamma^{H_0+V,H_0,\ldots,H_0}(f^{[n]})\right](V,\ldots,V)\cdot X\right\}\\
&\label{Thm-Pos-R1}=\int_{\R}f^{(n)}(x)\eta_{n,H_0,V,X}(x)dx.
\end{align}
Since $\eta_{n,H_0,V,X}$ is independent of the choice of $f$, choose $f\in C^{n}(\R)$ with $f^{(n)}\ge0$. This further implies $f^{[n]}\ge0$ on $\R^{n+1}$. Note that,
\[f(H_0+V)-f(H_0)=\left[\Gamma^{H_0+V,H_0}(f^{[1]})\right](V)=\left[\Gamma^{H_0,H_0+V}(f^{[1]})\right](V),\]
which further implies
\[\left[\Gamma^{H_0+V,H_0,\ldots,H_0}(f^{[n]})\right](V,\ldots,V)=\left[\Gamma^{H_0,H_0+V,H_0,\ldots,H_0}(f^{[n]})\right](V,\ldots,V).\]
By \cite[Proposition 2.6, Remark 2.7 (ii) and Definition 2.1]{LemSk20} we obtain
\begin{align}
\nonumber&\Tr\left(\left[\Gamma^{H_0,H_0+V,H_{0},\ldots,H_0}(f^{[n]})\right](V,\ldots,V)\cdot X\right)\\
\label{Thm-Pos-R2}&=\lim_{m\to\infty}\sum_{l_{1},\ldots,l_{n}\in\Z}f^{[n]}\left(\frac{l_{1}}{m},\ldots,\frac{l_{n}}{m},\frac{l_{1}}{m}\right)\Tr\left(E_{l_{1},m}VF_{l_{2},m}V\cdots VE_{l_{n},m}VE_{l_{1},m}\cdot X\right),
\end{align}
where $E_{l,m}:=E\left(\left[\frac{l}{m}\right.,\left.\frac{l+1}{m}\right)\right)$ and $F_{l,m}:=F\left(\left[\frac{l}{m}\right.,\left.\frac{l+1}{m}\right)\right)$ with $E$ and $F$ are the spectral measures of $H_0$ and $H_0+V$, respectively. Define,
\[\Sigma_{n,m}(\Omega_{1},\ldots,\Omega_{n})=\Tr\left(E(\Omega_{1})VF(\Omega_{2})V\cdots VE(\Omega_{n})VE(\Omega_{1})\cdot X\right),\]
where $\Omega_{1},\ldots,\Omega_{n}$ are Borel subsets of $\R$. Notice that from \eqref{Thm-Pos-R1} and \eqref{Thm-Pos-R2} the positivity of $\eta_{n,H_0,V,X}$ now only depends on the positivity of $\Sigma_{n,m}(\Omega_{1},\ldots,\Omega_{n})$.

Consider the case, $n=2$. Then, for any Borel subsets $\Omega_{1},\Omega_{2}$ of $\R$
\begin{align}
\nonumber\Sigma_{2,m}(\Omega_{1},\Omega_{2})&=\Tr\left(E(\Omega_{1})VF(\Omega_{2})VE(\Omega_{1})\cdot X\right)\\
\nonumber&=\Tr\left(E(\Omega_{1})\sqrt{X}VF(\Omega_{2})V\sqrt{X}E(\Omega_{1})\right)\ge0.
\end{align}
This establishes (A).

Now consider $n>2$ with $V$ as a rank one operator. Let $V=\langle\cdot,v_{1}\rangle v_{2}$, where either $v_{1}=v_{2}$ or $v_{1}=-v_{2}$ with $\|v_{1}\|=1$. Let $\{e_{j}\}_{j=1}^{\infty}$ be an orthonormal basis of $\hilh$. Then from the proof of \cite[Proposition 3.5]{Sk17Laa} we obtain
\begin{align}
\nonumber\Sigma_{n,m}(\Omega_{1},\ldots,\Omega_{n})&=\Tr\left(E(\Omega_{1})\sqrt{X}VF(\Omega_{2})V\cdots VE(\Omega_{n})V\sqrt{X}E(\Omega_{1})\right)\\
\nonumber&=(-1)^{(n-1)\cdot \delta_{-v_{1}}(v_{2})}\left\|F(\Omega_{2})v_{1}\right\|^{2}\prod_{k=3}^{n}\left\|E(\Omega_{k})v_{1}\right\|^{2}\\
\nonumber&\quad\times\sum_{j=1}^{\infty}\langle E(\Omega_{1})\sqrt{X}v_{2},e_{j}\rangle\langle \sqrt{X}E(\Omega_{1})e_{j},v_{1}\rangle\\
\nonumber&=(-1)^{n\cdot \delta_{-v_{1}}(v_{2})}\left\|F(\Omega_{2})v_{1}\right\|^{2}\prod_{k=3}^{n}\left\|E(\Omega_{k})v_{1}\right\|^{2}\left\|E(\Omega_{1})\sqrt{X}v_{1}\right\|^{2},
\end{align}
for all Borel subsets $\Omega_{1},\ldots,\Omega_{n}$ of $\R$, and $\delta_{-v_{1}}:\{v_{1},-v_{1}\}\to\{0,1\}$ is the indicator function. This implies the result (B).
\end{proof}

We conclude this section with the following result, which addresses the continuity and differentiability of the spectral shift functions (obtained in Theorems \ref{Thm-Mod-Sacase}, \ref{Thm-Mod-Unicase} and \ref{Thm-Mod-Concase}) under certain topologies, in the same spirit as \cite{CaGeLeNiPoSu16}.

\begin{ppsn}\label{cor:ssf_cont_diff}
Let $n\in\N$, $n\ge 2$ and $\epsilon>0$. Let $p=n+\epsilon$ and $q=1+\frac{n}{\epsilon}$. Consider $X\in\Sp^{q}(\hilh)$. Then the following assertions hold.
\begin{enumerate}[{\normalfont(i)}]
\item Let $H_0$ be a self-adjoint operator in $\hilh$, with $V\in\Sp^{p}_{sa}(\hilh)$. Then, the mappings $V\in\Sp_{sa}^{p}(\hilh)\mapsto\eta_{n,H_0,V,X}\in L^{1}(\R)$ and $X\in\Sp^{q}(\hilh)\mapsto\eta_{n,H_0,V,X}\in L^{1}(\R)$ are continuous in $L^{1}$-norm, where $\eta_{n,H_{0},V,X}$ is the $n$-th order spectral shift function obtained in Theorem \ref{Thm-Mod-Sacase}.
\item Let $U_0\in\mathcal{U}(\hilh)$, $A\in\Sp^{p}_{sa}(\hilh)$. Then, the mapping $X\in\Sp^{q}(\hilh)\mapsto\xi_{n,U_0,A,X}\in L^{1}(\cir)$ is continuous in $L^{1}(\cir)/\mathcal{P}_{n}(\cir)$-norm, where $\xi_{n,U_{0},A,X}$ is the $n$-th order spectral shift function obtained in Theorem \ref{Thm-Mod-Unicase}.
\item Let $T_{1},T_{0}\in\conth$ be such that $T_{1}-T_{0}\in\Sp^{p}(\hilh)$. Then, the map $X\in\Sp^{q}(\hilh)\mapsto\zeta_{n,T_{1},T_{0},X}\in L^{1}(\cir)$ is continuous in $L^{1}(\cir)/H^{1}(\cir)$-norm, where $\zeta_{n,T_{1},T_{0},X}$ is the $n$-th order spectral shift function obtained in Theorem \ref{Thm-Mod-Concase}.
\end{enumerate}
Additionally, if $t\in\R\mapsto X(t)\in\Sp^{q}(\hilh)$ is a smooth map such that $X'(t)=Y(t)$. Then $t\in\R\mapsto\eta_{n,H_{0},V,X(t)}\in L^{1}(\R)$, and  $t\in\R\mapsto\dot{\zeta}_{n,T_{1},T_{0},X(t)}\in L^{1}(\cir)/H^{1}(\cir)$ for $\zeta_{n,T_{1},T_{0},X(t)}\in L^{1}(\cir)$ are differentiable.
\end{ppsn}
	
\begin{proof}
Consider $V,W\in \Sp^p_{sa}(\hilh)$ and $X,Y\in\Sp^{q}(\hilh)$. Then for $n$-times differentiable function $f$ with bounded $f^{(n)}$, it follows from the proof of Theorem \ref{Thm-Mod-Sacase} that
\begin{align*}
&\left\|\eta_{n,H_0,V,X}-\eta_{n,H_0,W,X}\right\|_{L^{1}}\\
&=\sup_{\substack{f\in C_c^{n+1}(\R)\\\|f^{(n)}\|_{\infty,\R}\le1}}\left|\int_{\R}f^{(n)}(x)\left\{\eta_{n,H_0,V,X}(x)-\eta_{n,H_0,W,X}(x)\right\}dx\right|\\
&=\sup_{\substack{f\in C_c^{n+1}(\R)\\\|f^{(n)}\|_{\infty,\R}\le1}}\big|\Tr\left\{\big(\left[\Gamma^{H_0+V,(H_0)^{n}}(f^{[n]})\right](V,\ldots,V)-\left[\Gamma^{H_0+W,(H_0)^{n}}(f^{[n]})\right](W,\ldots,W)\big)\cdot X\right\}\big|\\
&=\sup_{\substack{f\in C_c^{n+1}(\R)\\\|f^{(n)}\|_{\infty,\R}\le1}}\bigg|\Tr\bigg\{\bigg(\left[\Gamma^{H_0+V,H_{0}+W,(H_0)^{n-1}}(f^{[n]})\right](V-W,V,\ldots,V)\\
&\hspace{1in}+\sum_{k=2}^{n}\left[\Gamma^{H_0+W,(H_0)^{n}}(f^{[n]})\right](W,\ldots,W,\underbrace{V-W}_\text{k-th place},V\ldots,V)\bigg)\cdot X\bigg\}\bigg|\\
&\hspace{4.3in}\big[\text{using \cite[Proposition 2.8]{Co22}}\big]\\
&\le C_{p,n}\,\|V-W\|_{p}\,\|X\|_{q}\left\{\|V\|_{p}^{n-1}+\sum_{k=2}^{n}\|W\|_{p}^{k-1}\,\|V\|_{p}^{n-k}\right\},\quad\big[\text{due to Theorem \ref{MOIEst-Self+Uni}}\big]
\end{align*}
for some constant $C_{p,n}>0$, and 
\begin{align*}
\left\|\eta_{n,H_0,V,X}-\eta_{n,H_0,V,Y}\right\|_{L^{1}}&=\sup_{\substack{f\in C_c^{n+1}(\R)\\\|f^{(n)}\|_{\infty,\R}\le1}}\left|\int_{\R}f^{(n)}(x)\left\{\eta_{n,H_0,V,X}(x)-\eta_{n,H_0,V,Y}(x)\right\}dx\right|\\
&\le \widetilde{C}_{p,n}\|V\|_{p}^{n}\|X-Y\|_{q}\quad\text{for some constant $\widetilde{C}_{p,n}>0$}.
\end{align*}
This proves (i). (ii) follows from \eqref{TMU-R10} and \eqref{TMU-R5}. 
Next consider,
\begin{align*}
&\left\|\frac{\eta_{n,H_0,V,X(t+h)}-\eta_{n,H_0,V,X(t)}}{h}-\eta_{n,H_0,V,Y(t)}\right\|_{L^1}\\
&=\sup_{\substack{f\in C_c^{n+1}(\R)\\\|f^{(n)}\|_{\infty,\R}\le1}}\left|\int_{\R}f^{(n)}(x)\left\{\frac{\eta_{n,H_0,V,X(t+h)}(x)-\eta_{n,H_0,V,X(t)}(x)}{h}-\eta_{n,H_0,V,Y(t)}(x)\right\}dx\right|\\
&=\sup_{\substack{f\in C_c^{n+1}(\R)\\\|f^{(n)}\|_{\infty,\R}\le1}}\left|\Tr\left\{\mathcal{R}^{\lin}_{n}\left(f,H_0,V\right)\cdot\left(\frac{X(t+h)-X(t)}{h}-Y(t)\right)\right\}\right|\\
&\leq c_n\|V\|_p^{n} \left\|\left(\frac{X(t+h)-X(t)}{h}-Y(t)\right)\right\|_q\quad[\text{by Theorem \ref{MOIEst-Self+Uni}}]\\
&\longrightarrow 0 \text{ as } h\to 0.
\end{align*}
This proves the last part of the corollary. The remaining cases are similar.
\end{proof}

\noindent\textit{Acknowledgment}: The authors thank the anonymous referees for a careful reading of the manuscript and for providing numerous insightful suggestions, which have greatly improved the clarity and exposition of the article. A. Chattopadhyay is supported by the Core Research Grant (CRG), File No: CRG/2023/004826, of SERB. S. Giri acknowledges the support by the Prime Minister's Research Fellowship (1902164), Government of India. C. Pradhan acknowledges support from the NBHM Postdoctoral Fellowship, Government of India, and from the Fulbright–Nehru Postdoctoral Fellowship.
\vspace{.1in}

\noindent\textit{Competing interests}: The authors declare none.


\begin{thebibliography}{10}

\bibitem{AlPe11}
A.B. Aleksandrov and V.V. Peller, {\it Trace formulae for perturbations of class $\mathcal{S}_m$}, J. Spectral Theory, {\bf 1} (2011), 1--26.

\bibitem{AlPe16}
A. B. Aleksandrov and V. V. Peller, {\it Krein's trace formula for unitary operators and operator Lipschitz functions} (Russian), Funktsional. Anal. i Prilozhen., {\bf 50} (2016), no. 3, 1--11. English Transl.: Funct. Anal. Appl., {\bf 50} (2016), no. 3, 167--175.

\bibitem{AlPe16Survey} 
A. B. Aleksandrov and V. V. Peller, {\it Operator Lipschitz functions} (Russian), Uspekhi Mat. Nauk, {\bf 71} (2016), no. 4(430), 3--106. English Transl.: Russian Math. Surveys, {\bf 71} (2016), no. 4, 605--702.

\bibitem{BhChGiPr23}
T. Bhattacharyya, A. Chattopadhyay, S. Giri and C. Pradhan, {\it Lipschitz estimates and an application to trace formulae}, Banach J. Math. Anal., {\bf 19} (2025), no. 4, Paper No. 68.

\bibitem{BiKr62}
M. S. Birman and M. G. Krein, {\it On the theory of wave operators and scattering operators}, Dokl. Akad. Nauk SSSR, {\bf 144} (1962), 475--478.

\bibitem{BiSo66} 
M. S. Birman and M. Z. Solomyak, {\it Double Stieltjes operator integrals}, Problems of mathematical physics, No. I: Spectral theory and wave processes (Russian), Izdat. Leningrad. Univ., Leningrad, {\bf 1} (1966), 33--67. English Transl.: Topics Math. Physics, {\bf 1} (1967), 25--54, Consultants Bureau Plenum Publishing Corporation, New York.

\bibitem{BiYa93} 
M. S. Birman and D. R. Yafaev, {\it The spectral shift function. The papers of M. G. Krein and their further development} (Russian), Algebra i Analiz, {\bf 4} (1992), no. 5, 1--44. English Transl.: St. Petersburg Math. J., {\bf 4} (1993), no. 5, 833--870.

\bibitem{CaGeLeNiPoSu16}
A. Carey, F. Gesztesy, G. Levitina, R. Nichols, D. Potapov and F. Sukochev, {\it Double operator integral methods applied to continuity of spectral shift functions}, { J. Spectr. Theory}, {\bf 6} (2016), no. 4, 747--779.

\bibitem{ChCo97} 
A. H. Chamseddine and A. Connes, {\it The spectral action principle}, Comm. Math. Phys., {\bf 186} (1997), no. 3, 731--750.

\bibitem{ChCoGiPr24} 
A. Chattopadhyay, C. Coine, S. Giri and C. Pradhan, {\it Higher order $\Sp^p$-differentiability: The unitary case}, J. Spectr. Theory, {\bf 15} (2025), no. 1, 195--222.

\bibitem{ChDaPr24}
A. Chattopadhyay, S. Das and C. Pradhan, {\it Second-order trace formulae}, Math. Nachr., {\bf 297} (2024), no. 7, 2581--2608.
	
\bibitem{ChPrSk24}
A. Chattopadhyay, C. Pradhan and A. Skripka, {\it Higher-order trace formulas for contractive and dissipative operators}, Canad. J. Math., (2025), pp. 1--26. \url{https://doi.org/10.4153/S0008414X25101296}

\bibitem{Co22} 
C. Coine, {\it Perturbation theory and higher order $\Sp^p$-differentiability of operator functions}, Rev. Mat. Iberoam., {\bf 38} (2022), no. 1, 189--221.

\bibitem{Co24} 
C. Coine, {\it Functions of unitaries with $\Sp^{p}$-perturbations for non continuously differentiable functions}, Studia Math., {\bf 286} (2026), no. 3, 207--242.

\bibitem{CLPST1}
C. Coine, C. Le Merdy, D. Potapov, F. Sukochev and A. Tomskova, {\it Resolution of Peller's problem concerning Koplienko-Neidhardt trace formulae}, Proc. Lond. Math. Soc. (3), {\bf 113} (2016), no. 2, 113--139. 

\bibitem{CLPST2} 
C. Coine, C. Le Merdy, D. Potapov, F. Sukochev and A. Tomskova, {\it Peller's problem concerning Koplienko-Neidhardt trace formulae: the unitary case}, J. Funct. Anal., {\bf 271} (2016), no. 7, 1747--1763.
	
\bibitem{CoLemSkSu19} 
C. Coine, C. Le Merdy, A. Skripka and F. Sukochev, {\it Higher order $\Sp^2$-differentiability and application to Koplienko trace formula}, J. Funct. Anal., {\bf 276} (2019), no. 10, 3170--3204.
	
\bibitem{CoLemSu21} 
C. Coine, C. Le Merdy and F. Sukochev, {\it When do triple operator integrals take value in the trace class?}, Ann. Inst. Fourier (Grenoble), {\bf 71} (2021), no. 4, 1393--1448.

\bibitem{Connes-Book} 
A. Connes, {\it Noncommutative geometry}, Academic Press, Inc., San Diego, CA, 1994, xiv+661 pp.
		
\bibitem{Conway-Book} 
J. B. Conway, {\it A course in operator theory}, Graduate Studies in Mathematics, vol. 21., American Mathematical Society, Providence, RI, 2000, xvi+372 pp.

\bibitem{DaKr56} 
Yu. L. Daletskii and S. G. Krein, {\it Integration and differentiation of functions of Hermitian operators and applications to the theory of perturbations} (Russian), Vorone\v z. Gos. Univ. Trudy Sem. Funkcional. Anal., {\bf 1956} (1956), no. 1, 81--105. English Transl.: Amer. Math. Soc. Transl. Series 2, {\bf 47} (1965), 1--30.

\bibitem{DeLo93} 
R. A. DeVore and G. G. Lorentz, {\it Constructive approximation}, Grundlehren der mathematischen Wissenschaften [Fundamental Principles of Mathematical Sciences], vol. 303. Springer-Verlag, Berlin, 1993, x+449 pp.

\bibitem{DiUhl77}
J. Diestel and J. J. Uhl, {\it Vector measures}, Mathematical Surveys, vol. 15, American Mathematical Society, Providence, RI, 1977, xiii+322 pp.
	
\bibitem{KiPoShSu14}
E. Kissin, D. Potapov, V. Shulman and F. Sukochev, {\it Operator smoothness in Schatten norms for functions of several variables: Lipschitz conditions, differentiability and unbounded derivations}, Proc. Lond. Math. Soc. (3), {\bf 105} (2012), no. 4, 661--702.

\bibitem{Ko84} 
L. S. Koplienko, {\it The trace formula for perturbations of nonnuclear type} (Russian), Sibirsk. Mat. Zh., {\bf 25} (1984), no. 5, 62--71. English Transl.: Siberian Math. J., {\bf 25} (1984), 735--743.

\bibitem{Kr53} 
M. G. Krein, {\it On the trace formula in perturbation theory} (Russian), Mat. Sbornik N.S., {\bf 33/75} (1953), 597--626.

\bibitem{Kr62} 
M. G. Krein, {\it On perturbation determinants and a trace formula for unitary and self-adjoint operators} (Russian), Dokl. Akad. Nauk SSSR, {\bf 144} (1962), 268--271. English Transl.: Soviet Math. Dokl., {\bf 3} (1962), 707--710.

\bibitem{La65}
H. Langer, {\it Eine Erweiterung der Spurformel der St\"{o}rungstheorie}, Math. Nachr., {\bf 30} (1965), 123--135.

\bibitem{LemSk20} 
C. Le Merdy and A. Skripka, {\it Higher order differentiability of operator functions in Schatten norms}, J. Inst. Math. Jussieu, {\bf 19} (2020), no. 6, 1993--2016.

\bibitem{Li52} 
I. M. Lifshitz, {\it On a problem of the theory of perturbations connected with quantum statistics} (Russian), Uspehi Matem. Nauk (N.S.), {\bf 7} (1952), no. 1(47), 171--180.

\bibitem{MaNe15}
M. M. Malamud and H. Neidhardt, {\it Trace formulas for additive and non-additive perturbations}, Adv. Math., {\bf 274} (2015), 736--832.

\bibitem{MaNePe16} 
M. M. Malamud, H. Neidhardt and V. V. Peller, {\it Analytic operator Lipschitz functions in the disk and a trace formula for functions of contractions} (Russian), Funktsional. Anal. i Prilozhen., {\bf 51} (2017), no. 3, 33--55. English Transl.: Funct. Anal. Appl., {\bf 51} (2017), no. 3, 185--203.

\bibitem{MaNePe19}
M. M. Malamud, H. Neidhardt and V. V. Peller, {\it Absolute continuity of spectral shift}, J. Funct. Anal., {\bf 276} (2019), no. 5, 1575--1621.
	
\bibitem{Na40} 
M. A. Naimark, {\it Spectral functions of symmetric operator} (Russian), Izvestia Akad. Nauk SSSR, Ser. Matem., {\bf 4} (1940), 277--318.

\bibitem{Ne88-II}
H. Neidhardt, {\it Scattering matrix and spectral shift of the nuclear dissipative scattering theory. II}, J. Operator Theory, {\bf 19} (1988), no. 1, 43--62.

\bibitem{Ne88}
H. Neidhardt, {\it Spectral shift function and Hilbert-Schmidt perturbation: extensions of some work of L. S. Koplienko}, Math. Nachr., {\bf 138} (1988), 7--25.

\bibitem{Pa69}
B. S. Pavlov, {\it Multidimensional operator integrals}, Problems of Math. Anal., No. 2: Linear Operators and Operator Equations (Russian), Izdat. Leningrad. Univ., Leningrad, (1969), 99--122.

\bibitem{Pe85}
V. V. Peller, {\it Hankel operators in the theory of perturbations of unitary and selfadjoint operators} (Russian), Funktsional. Anal. i Prilozhen., {\bf 19} (1985), no. 2, 37--51, 96.

\bibitem{Pe87}
V.V. Peller, {\it For which f does $A-B \in \mathcal{S}^p$ imply that $f(A)-f(B) \in \mathcal{S}^p$?}, Oper. Theory Adv. Appl., vol. 24,
Birkhäuser, Basel, 1987, pp. 289--294.

\bibitem{Pe05} 
V. V. Peller, {\it An extension of the Koplienko-Neidhardt trace formulae,} J. Funct. Anal., {\bf 221} (2005), no. 2, 456--481.

\bibitem{Pe06}
V. V. Peller, {\it Multiple operator integrals and higher operator derivatives}, J. Funct. Anal., {\bf 233} (2006), no. 2, 515--544.

\bibitem{Pe09}
V. V. Peller, {\it Differentiability of functions of contractions}, Linear and Complex Analysis, Amer. Math. Soc. Transl. Ser. 2, vol. 226, Amer. Math. Soc., Providence, RI, 2009, 109--131 pp.

\bibitem{Pe16}
V. V. Peller, {\it The Lifshitz-Krein trace formula and operator Lipschitz functions}, Proc. Amer. Math. Soc., {\bf144} (2016), no. 12, 5207--5215.

\bibitem{PoSkSu13}
D. Potapov, A. Skripka and F. Sukochev, {\it Spectral shift function of higher order},  Invent. Math., {\bf 193} (2013), no. 3, 501--538.

\bibitem{PoSkSu14} 
D. Potapov, A. Skripka and F. Sukochev, {\it Higher-order spectral shift for contractions}, Proc. Lond. Math. Soc. (3), {\bf 108} (2014), no. 2, 327--349.

\bibitem{PoSkSu16}
D. Potapov, A. Skripka and F. Sukochev, {\it Functions of unitary operators: derivatives and trace formulas}, J. Funct. Anal., {\bf 270} (2016), no. 6, 2048--2072.

\bibitem{PoSkSuTo17}
D. Potapov, A. Skripka, F. Sukochev and A. Tomskova, {\it Multilinear Schur multipliers and Schatten properties of operator Taylor remainders}, Adv. Math., {\bf 320} (2017), 1063--1098.

\bibitem{PoSu12}
D. Potapov and F. Sukochev, {\it Koplienko spectral shift function on the unit circle}, Comm. Math. Phys., {\bf 309} (2012), no. 3, 693--702.
	
\bibitem{Sk17Adv} 
A. Skripka, {\it Estimates and trace formulas for unitary and resolvent comparable perturbations}, Adv. Math., {\bf 311} (2017), 481--509.

\bibitem{Sk17Laa}
A. Skripka, {\it On positivity of spectral shift functions}, Linear Algebra Appl., {\bf 523} (2017), 118--130.
	
\bibitem{SkTo19}
A. Skripka and A. Tomskova, {\it Multilinear operator integrals}, Lecture Notes in Mathematics. Theory and applications, vol. 2250, Springer, Cham, 2019, xi+190 pp.

\bibitem{Vo79}
D. Voiculescu, {\it Some results on norm-ideal perturbations of Hilbert space operators}, J. Operator Theory, {\bf 2} (1979), no. 1, 3--37. 

\bibitem{Yafaev} 
D. R. Yafaev, {\it Mathematical scattering theory. General theory}, Translations of Mathematical Monographs, vol. 105., American Mathematical Society, Providence, RI, 1992, x+341 pp.
	
\end{thebibliography}
\end{document}